\pdfoutput=1
\documentclass[reqno]{amsart}
\DeclareFontEncoding{OT2}{}{}
\DeclareFontSubstitution{OT2}{cmr}{m}{n}
\DeclareFontFamily{OT2}{cmr}{\hyphenchar\font45 }
\DeclareFontShape{OT2}{cmr}{m}{n}{
<5><6><7><8><9>gen*wncyr
<10><10.95><12><14.4><17.28><20.74><24.88>wncyr10}{}
\DeclareFontShape{OT2}{cmr}{b}{n}{
<5><6><7><8><9>gen*wncyb
<10><10.95><12><14.4><17.28><20.74><24.88>wncyb10}{}
\DeclareMathAlphabet{\mathcyr}{OT2}{cmr}{m}{n}
\DeclareMathAlphabet{\mathcyb}{OT2}{cmr}{b}{n}
\SetMathAlphabet{\mathcyr}{bold}{OT2}{cmr}{b}{n}
\usepackage{geometry}
\usepackage{fancybox,ascmac,comment}
\usepackage{amssymb}
\usepackage{amsfonts}
\usepackage{amsmath}
\usepackage{amsthm}
\usepackage{mathtools}
\usepackage{enumerate}

\usepackage{mleftright}
\mleftright
\usepackage[all]{xy}
\mathtoolsset{showonlyrefs=true}
\numberwithin{equation}{section}
\newcommand{\sh}{\mathbin{\mathcyr{sh}}}
\renewcommand{\bf}{\mathbf{f}}
\newcommand{\bh}{\mathbf{h}}
\newcommand{\bk}{\mathbf{k}}
\newcommand{\bl}{\mathbf{l}}
\newcommand{\bs}{\mathbf{s}}
\newcommand{\bt}{\mathbf{t}}
\newcommand{\bp}{\boldsymbol{p}}
\newcommand{\bx}{\mathbf{x}}
\newcommand{\etab}{\boldsymbol{\eta}}
\newcommand{\emp}{\varnothing}
\newcommand{\bbZ}{\mathbb{Z}}
\newcommand{\bbQ}{\mathbb{Q}}
\newcommand{\bbR}{\mathbb{R}}
\newcommand{\bbC}{\mathbb{C}}
\newcommand{\cA}{\mathcal{A}}
\newcommand{\cF}{\mathcal{F}}
\newcommand{\cP}{\mathcal{P}}
\newcommand{\cR}{\mathcal{R}}
\newcommand{\cS}{\mathcal{S}}
\newcommand{\cZ}{\mathcal{Z}}
\newcommand{\cRS}{\mathcal{RS}}
\newcommand{\hcA}{\widehat{\mathcal{A}}}
\newcommand{\hcS}{\widehat{\mathcal{S}}}
\newcommand{\hcF}{\widehat{\mathcal{F}}}
\newcommand{\hcRS}{\widehat{\mathcal{RS}}}
\newcommand{\fmpl}{\text{\rm\pounds}}
\newcommand{\Li}{\mathrm{Li}}
\newcommand{\fH}{\mathfrak{H}}

\newcommand{\tp}{\tilde{p}}
\newcommand{\tq}{\tilde{q}}
\newcommand{\tr}{\tilde{r}}
\newcommand{\dch}{\mathrm{dch}}

\newcommand{\rt}{\mathrm{rt}}
\renewcommand{\mod}{\mathrm{mod}}
\newcommand{\tzero}{\tilde{0}}
\newcommand{\tone}{\tilde{1}}
\newcommand{\har}{\mathrm{har}}
\newcommand{\h}{\mathrm{h}}

\newcommand{\vv}[2]{\begin{pmatrix}#1\\ #2\end{pmatrix}}
\newcommand{\jump}[1]{\ensuremath{[\![#1]\!]}}
\newtheorem{theorem}{Theorem}[section]
\newtheorem{proposition}[theorem]{Proposition}
\newtheorem{lemma}[theorem]{Lemma}
\newtheorem{corollary}[theorem]{Corollary}

\theoremstyle{definition}
\newtheorem{definition}[theorem]{Definition}
\theoremstyle{remark}
\newtheorem{remark}[theorem]{Remark}
\newtheorem{example}[theorem]{Example}
\title{Iterated Integrals associated with colored rooted trees}
\author{Hanamichi Kawamura}
\address[Hanamichi Kawamura]{Department of Mathematics, Faculty of Science Division I, Tokyo University of Science, 1-3 Kagurazaka, Shinjuku-ku, Tokyo, 162-8601, Japan}
\email{1121026@ed.tus.ac.jp}
\allowdisplaybreaks
\begin{document}
\begin{abstract}
  In this paper, we introduce iterated integrals associated with colored rooted trees and give proofs for the shuffle relations for $\bp$-adic finite and $t$-adic symmetric polylogarithms. This method generalizes the theory of the finite multiple zeta values associated with $2$-colored rooted trees introduced by Ono.
\end{abstract}
\maketitle
\section{Introduction}
A \emph{$2$-colored rooted tree} is introduced by Ono \cite{ono17} as a tree (in the graph-theoretic meaning) equipped with certain data about coloring and distinguishing vertices. The purpose of introducing them in \cite{ono17} is to prove the shuffle relation for finite multiple zeta values. Recently, by Ono--Seki--Yamamoto \cite{osy21}, this theory is found applicable for the shuffle relation of ($t$-adic) symmetric multiple zeta values. Moreover, Ono \cite{ono22} also proved some purely algebraic identities by using the theory of $2$-colored rooted trees. 
Besides, the shuffle relations for finite and symmetric multiple zeta values can be also proved by their expressions as iterated integrals (\cite{hirose20} for the symmetric case). From this perspective, we generalize the theory of finite and symmetric multiple zeta values associated with $2$-colored colored rooted trees to general iterated integrals. The following are the main result of this paper, which generalizes the $\bp$-adic shuffle relation \cite[Theorem 6.4]{seki17}, \cite[Proposition 2.3.3 (ii)]{jarossay19} and $t$-adic shuffle relation \cite[(1.3)]{osy21}, \cite[Proposition 2.3.1]{jarossay19} to multiple polylogarithms. See \S\ref{sec_mlvs} for precise definitions of the symbols appearing below.
\begin{theorem}[$=$ Theorem \ref{thm_p_shuffle} for $\cF=\cA$, Theorem \ref{thm_t_shuffle} for $\cF=\cS$]
  Let $N,k_{1},\ldots,k_{r},l_{1},\ldots,l_{s}$ be positive integers, $x_{1},\ldots,x_{r},y_{1},\ldots,y_{s}$ indeterminates, $\alpha\in\bbZ/N\bbZ$ and $\cF$ either $\cA$ or $\cS$. Put
  \[\Lambda_{\cF}\coloneqq\begin{cases}\bp & \text{if }\cF=\cA,\\ t & \text{if } \cF=\cS,\end{cases}\qquad\text{and}\qquad d_{\cF}\coloneqq\begin{cases}\bp & \text{if } \cF=\cA,\\ \alpha & \text{if }\cF=\cS.\end{cases}\]
  Then we have
  \begin{align}
    &\fmpl_{\hcF,\alpha}\left(\vv{x_{1},\ldots,x_{r}}{k_{1},\ldots,k_{r}}\sh\vv{y_{1},\ldots,y_{s}}{l_{1},\ldots,l_{s}}\right)\\
    &\qquad=(-1)^{l_{1}+\cdots+l_{s}}y_{1}^{d_{\cF}}\sum_{f_{1},\ldots,f_{s}\ge 0}\left(\prod_{i=1}^{s}\binom{l_{i}+f_{i}-1}{f_{i}}\right)\fmpl_{\hcF,\alpha}\vv{x_{1}/y_{1},\ldots,x_{r}/y_{1},1/y_{1},y_{s}/y_{1},\ldots,y_{2}/y_{1}}{k_{1},\ldots,k_{r},l_{s}+f_{s},\ldots,l_{1}+f_{1}}\Lambda_{\cF}^{f_{1}+\cdots+f_{s}}.
\end{align}
Here the power $y_{1}^{\alpha}$ means $y_{1}^{n}$ by using the unique element $n\in\{0,1,\ldots,N-1\}$ in the equivalent class $\alpha$.
\end{theorem}
This paper is organized as follows: \S\ref{sec_def} is devoted to recall basic preliminaries about $2$-colored rooted trees and the associated zeta values. In \S\ref{sec_ii}, we introduce $D$-colored rooted trees and the associated iterated integrals. Each property appearing in \S\ref{sec_def} is generalized to the iterated integrals in \S\ref{sec_ii} and following \S\ref{sec_mlvs}. The last section also provides some applications for the $\bp$-adic and $t$-adic multiple polylogarithms and $L$-values.
\section*{Acknowledgements}
The author would like to express his sincere gratitude to Prof.~Masataka Ono, Prof.~Shin-ichiro Seki, Yoshiyuki Miyamori, Yuta Tanaka, Taiga Ozaki, Junichi Yokoyama and Honami Sakamoto for their careful reading of the manuscript. He is also thankful to Taizen Suzuki and Takumi Maesaka for their helpful comments.
\section{Definitions and known results}\label{sec_def}
In this paper, a \emph{tree} means a finite connected acyclic graph. In a tree $(V,E)$, if an edge $e\in E$ connects vertices $v$ and $w$ then we write $e=\{v,w\}$. For a vertices $v$ and $w$, denote the path between them by $P(v,w)$. The \emph{degree} $\deg(v)$ of a vertex $v$ in a tree is a number of edges connected to $v$. A \emph{terminal} (resp.~\emph{branched vertex}) is a vertex whose degree is $1$ (resp.~greater than $2$). For a vertex $v$, we put $N_{v}\coloneqq\{w\in V\mid\{v,w\}\in E\}$ and say that its element is \emph{adjacent} to $v$. Here and in what follows, when more than one trees appear simultaneously, we assume that they are mutually disjoint unless noted otherwise.
\begin{definition}[{\cite[Definition 1.2]{ono17}}, {\cite[Definition 3.14]{osy21}}]\label{def_2crt}\phantom{}
  \begin{enumerate}[(1)]
    \item A \emph{$2$-colored rooted tree} $X=(V,E,\rt,V_{\bullet})$ consists of below data:
  \begin{itemize}
    \item a tree $(V,E)$.
    \item a vertex $\rt\in V$, called the \emph{root} of $X$.
    \item a subset $V_{\bullet}\subseteq V$ including all terminals of $(V,E)$.
  \end{itemize}
  A vertex or edge of $X=(V,E,\rt,V_{\bullet})$ means them of the ``underlying'' tree $(V,E)$ of $X$. We write $V_{\circ}\coloneqq V\setminus V_{\bullet}$ for such $X$. For a given vertex $v\in V$, the \emph{parent} $p_{v}$ of $v$ is the adjacent vertex of $v$ which is nearest to the root.
  \item\label{def_2crt_index} An \emph{index} on a $2$-colored rooted tree $X=(V,E,\rt,V_{\bullet})$ is a map $\bk\colon E\to\bbZ_{\ge 0}$. We usually denote $\bk(\{v,p_{v}\})$ by $\bk_{v}$ for each vertex $v$ (we put $\bk_{\rt}\coloneqq 1$ for later convenience). We say such an index is \emph{essentially positive} if $\sum_{e\in P(v,w)}\bk(e)$ is positive for all $v$, $w\in V_{\bullet}$ ($v\neq w$).  
  \item A $2$-colored rooted tree $X=(V,E,\rt,V_{\bullet})$ is called a \emph{linear $2$-colored rooted tree} (or simply \emph{linear tree}) if it satisfies all of
  \begin{itemize}
    \item the root is a terminal of $V$.
    \item the degree of a vertex is always $1$ or $2$.
    \item $V_{\bullet}=V$.
  \end{itemize}
  \item\label{def_harvestability} We say that a pair $(X;\bk)$ consisting of a $2$-colored rooted tree $X=(V,E,\rt,V_{\bullet})$ and an index $\bk$ on $X$ is \emph{harvestable} if all of the following conditions are satisfied:
  \begin{itemize}
    \item the root is a terminal (thus $\rt\in V_{\bullet}$).
    \item $\deg(v)\le 2$ (resp.~$\ge 3$) for any $v\in V_{\bullet}$ (resp.~$v\in V_{\circ}$).
    \item For $v\in V$, if $p_{v}\in V_{\circ}$ then $\bk_{v}>0$.
    \item For $v\in V_{\bullet}$, if $p_{v}\in V_{\bullet}$ then $\bk_{v}>0$.
  \end{itemize}
  It is clear that a harvestable pair $(X;\bk)=(V,E,\rt,V_{\bullet})$ with a linear tree $X$ is uniquely represented as $V=\{v_{1},\ldots,v_{r+1}\}$, $E=\{\{v_{i},v_{i+1}\}\mid i=1,\ldots,r\}$, $\rt=v_{r+1}$, $V_{\bullet}=V$ and $\bk=(k_{1},\ldots,k_{r})$ with the positive integers $k_{i}\coloneqq\bk_{v_{i}}$. 
\end{enumerate}
\end{definition}
\begin{remark}
  Note that the requirements in \eqref{def_harvestability} are same as that of \cite[Definition 3.14]{osy21}, not that of \cite[Definition 2.6]{ono17} (i.e., the index part of a harvestable pair should be always essentially positive. See \cite[Remark 3.15]{osy21}). 
\end{remark}
When we indicate a $2$-colored rooted tree as a diagram, we use the symbol $\bullet$ (resp.~$\circ$) for an element of $V_{\bullet}$ (resp.~$V_{\circ}$). The square symbol ($\blacksquare$ or $\square$) stands for the root. For instance, the following diagram indicates a linear tree:
\[\begin{xy}
  {(0,0)*{\blacksquare} \ar @{-} (10,0)*{\bullet}},
  {(10,0)*{\bullet} \ar @{-} (16,0)*{}},
  {(16,0)*{} \ar @{.} (30,0)*{}},
  {(30,0)*{} \ar @{-} (36,0)*{\bullet}},
  {(36,0)*{\bullet} \ar @{-} (46,0)*{\bullet}},
  {(0,2.5)*{\rt}},
  {(10,2)*{v_{r}}},
  {(36,2)*{v_{2}}},
  {(46,2)*{v_{1}}}
  \end{xy}\]
\begin{definition}
Let $\{X_{i}\}_{i}=\{(V_{i},E_{i})\}_{i=1,\ldots,n}$ be a family of mutually disjoint trees. Choose vertices $v_{1},\ldots,v_{n}$ from each tree $X_{1},\ldots,X_{n}$. Then we say that the graph
\[(V,E)\coloneqq\left(\{v\}\cup\bigcup_{i=1}^{n}(V_{i}\setminus\{v_{i}\}),\bigcup_{i=1}^{n}\{\{v,w\}\mid w\in N_{v_{i}}\}\cup(E_{i}\setminus\{\{v_{i},w\}\mid w\in N_{v_{i}}\})\right)\]
is obtained
by \emph{grafting $X_{1},\ldots,X_{n}$ along $v_{1},\ldots,v_{n}$} and write $G_{v_{1},\ldots,v_{n}}(X_{1},\ldots,X_{n})$. Here $v$ is a new vertex where $v_{1},\ldots,v_{n}$ are gathered and we call $v$ the \emph{cluster vertex}. When $n=1$, we put $G_{v}(X)=X$ for any vertex $v$. 
\end{definition}
For given rooted trees $X_{1},\ldots,X_{n}$, there is a canonical grafting by applying the above procedures along their roots. Then we omit subscripts as $G(X_{1},\ldots,X_{n})$.
For a family $\{X_{i}\}_{i}$ of $2$-colored rooted trees, we can also graft them by choosing a new root from $G(\{X_{i}\}_{i})$ (this $X_{i}$ stands for the underlying tree of $X_{i}$ by abuse of notation) and deciding whether to let the cluster vertex be in $V_{\bullet}$ or $V_{\circ}$ (the choice of cluster vertex and colors of other vertices are already given).

The following lemma is easy but gives a useful characterization of the harvestability. We prove it in Lemma \ref{p_harvestability_characterization} in a more general setting.
\begin{lemma}[{\cite[Remark 2.8]{ono17}}]\label{harvestability_characterization}
  Let $X=(V,E,\rt,V_{\bullet})$ be a $2$-colored rooted tree and $\bk$ an index on $X$ such that the pair $(X;\bk)$ is harvestable. Then there uniquely exists a finite non-empty set of pairs $\{(X_{i};\bk^{(i)})\}_{i}=\{(V_{i},E_{i},\rt_{i},V_{i,\bullet};\bk^{(i)})\}_{i=0,\ldots,r}$ such that all of the following conditions are satisfied:
\begin{itemize}
  \item $(X_{1};\bk^{(1)}),\ldots,(X_{r};\bk^{(r)})$ are harvestable.
  \item $X_{0}$ is linear.
  \item $\bk^{(0)}(e)>0$ for $e\in E_{0}\setminus\{\{\rt_{0},u_{0}\}\}$ $(u_{0}\in V_{0}\setminus\{\rt_{0}\}$ is nearest to $\rt_{0})$.
  \item $\rt$ is the farthest vertex from $\rt_{0}$ in $X_{0}$.
  \item $(V,E)=G(X_{0},\ldots,X_{r})$ and the cluster vertex of this grafting is in $V_{\circ}$.
\end{itemize}
The last condition is indicated as follows.
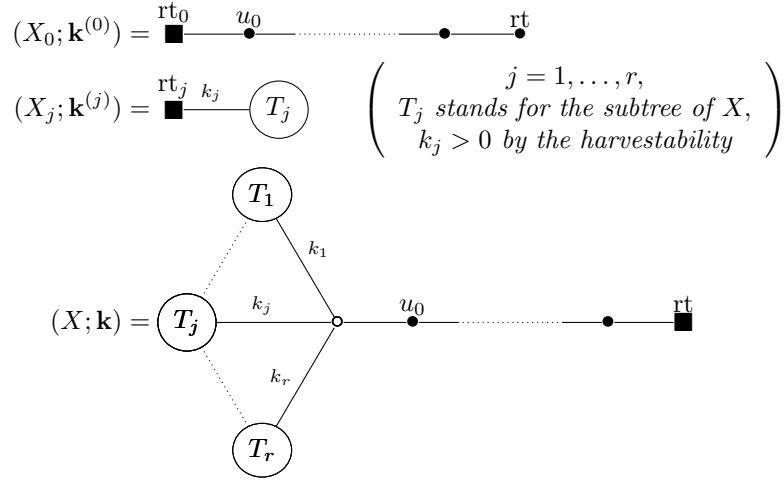
\begin{align}
  (X_{0};\bk^{(0)})
  &=\begin{xy}
  {(0,0)*{\blacksquare} \ar @{-} (10,0)*{\bullet}},
  {(10,0)*{\bullet} \ar @{-} (16,0)*{}},
  {(16,0)*{} \ar @{.} (30,0)*{}},
  {(30,0)*{} \ar @{-} (36,0)*{\bullet}},
  {(36,0)*{\bullet} \ar @{-} (46,0)*{\bullet}},
  {(0,3)*{\rt_{0}}},
  {(10,2)*{u_{0}}},
  {(46,2)*{\rt}}
  \end{xy}\\
  (X_{j};\bk^{(j)})
  &=\begin{xy}
    {(0,0)*{\blacksquare} \ar @{-}^{k_{j}} (10,0)*{}},
    {(14,0)*++[o][F]{T_{j}}},
    {(0,3)*{\rt_{j}}}
    \end{xy}\qquad\left(\begin{array}{c}j=1,\ldots,r,\\T_{j}\text{ stands for the subtree of }X,\\ k_{j}>0\text{ by the harvestability}\end{array}\right)\\
  (X;\bk)&=\begin{xy}
    {(0,0)*{\circ} \ar @{-} (10,0)*{\bullet}},
    {(10,0)*{\bullet} \ar @{-} (16,0)*{}},
    {(16,0)*{} \ar @{.} (30,0)*{}},
    {(30,0)*{} \ar @{-} (36,0)*{\bullet}},
    {(36,0)*{\bullet} \ar @{-} (46,0)*{\blacksquare}},
    {(0,0)*{\circ} \ar @{-}_{k_{1}} (-10,17)*++[o][F]{T_{1}}},
    {(0,0)*{\circ} \ar @{-}_{k_{j}} (-20,0)*++[o][F]{T_{j}}},
    {(0,0)*{\circ} \ar @{-}_{k_{r}} (-10,-17)*++[o][F]{T_{r}}},
    {(-10,17)*++[o][F]{T_{1}} \ar @{.} (-20,0)*++[o][F]{T_{j}}},
    {(-10,-17)*++[o][F]{T_{r}} \ar @{.} (-20,0)*++[o][F]{T_{j}}},
    {(10,2)*{u_{0}}},
    {(46,2.5)*{\rt}}
    \end{xy}
\end{align}
\end{lemma}
Define $\fH$ as the non-commutative polynomial ring $\bbQ\langle e_{0},e_{1}\rangle$ and $\fH^{1}\coloneqq\bbQ\oplus e_{1}\fH$. Endow the $\bbQ$-bilinear \emph{shuffle product} $\sh$ on $\fH$ recursively as
\begin{gather}
  1\sh w=w\sh 1=w,\\
  we_{i}\sh w'e_{j}=(we_{i}\sh w')e_{j}+(w\sh w'e_{j})e_{i}
\end{gather} 
for $w$, $w'\in\fH$ and $i$, $j\in\{0,1\}$. This product gives $\fH^{1}$ a commutative $\bbQ$-algebra structure \cite{reutenauer93}.
\begin{definition}[{\cite[Definition 3.16]{osy21}}]\label{word_with_tree}
  We assign an element $O_{(X;\bk)}\in\fH^{1}$ for each harvestable pair $(X;\bk)=(V,E,\rt,V_{\bullet};\bk)$ by the following inductive rules.
  \begin{enumerate}
    \item When $X$ is a linear tree, we can take the vertex $u_{1}$ farthest from $\rt$. Then let $u_{i+1}$ be the parent of $u_{i}$, $r$ the positive integer determined by $u_{r+1}=\rt$ and
    \[O_{(X;\bk)}\coloneqq\prod_{i=1}^{r}e_{1}e_{0}^{\bk(\{u_{i},u_{i+1}\})-1}.\]
    \item Assume that $(X;\bk)$ is a general harvestable pair which is not linear and use the same symbols as in Lemma \ref{harvestability_characterization}. Then we define
    \[O_{(X;\bk)}\coloneqq (O_{(X_{1};\bk^{(1)})}\sh\cdots\sh O_{(X_{r};\bk^{(r)})})e_{0}^{\bk^{(0)}(\{\rt_{0},u_{0}\})}O_{(V_{0}\setminus\{\rt_{0}\},E_{0}\setminus\{\rt_{0},u_{0}\},\rt,V_{0}\setminus\{\rt_{0}\};\bk^{(0)}|_{E_{0}\setminus\{\{\rt_{0},u_{0}\}\}})}.\]
  \end{enumerate}
\end{definition}
\begin{definition}[{\cite[Definition 1.3]{ono17}}]\label{def_fmzv_2crt}
  Let $X=(V,E,\rt,V_{\bullet})$ be a $2$-colored rooted tree and $\bk$ an index on $X$. Then we put
  \[\zeta_{<M}(X;\bk)\coloneqq\sum_{\substack{(m_{v})_{v}\in\bbZ_{\ge 1}^{V_{\bullet}}\\ \sum_{v\in V_{\bullet}}m_{v}=M}}\prod_{e\in E}\left(\sum_{e\in P(\rt,v)}m_{v}\right)^{-\bk(e)},\]
  for a positive integer $M$, and define the \emph{finite multiple zeta value} (FMZV for short) \emph{associated with $X$} as
  \[\zeta_{\cA}(X;\bk)\coloneqq(\zeta_{<p}(X;\bk)~\mod~p)_{p\in\cP}\in\cA.\]
  Here $\cA$ denotes the \emph{ring of integers modulo infinitely large primes}
  \[\cA\coloneqq\left(\prod_{p\in\cP} \bbZ/p\bbZ\right)\Biggm/\left(\bigoplus_{p\in\cP} \bbZ/p\bbZ\right),\]
  where $\cP$ is the set of all prime numbers. In particular, when $X$ is linear and $(X;\bk)$ is a harvestable pair, there exist positive integers $k_{1},\ldots,k_{r}$ (see the last assertion in Definition \ref{def_2crt}) such that the associated value is written as
  \[\zeta_{\cA}(X;\bk)=\left(\sum_{0<n_{1}<\cdots<n_{r}<p}\frac{1}{n_{1}^{k_{1}}\cdots n_{r}^{k_{r}}}~\mod~p\right)_{p\in\cP}\in\cA.\]
  The right-hand side is usually denoted as $\zeta_{\cA}(k_{1},\ldots,k_{r})$ and called the \emph{finite multiple zeta value} in the usual sense \cite{kz22}.  
\end{definition}
\begin{definition}[{\cite[Definitions 3.2, 3.19]{osy21}}]\label{def_t_smzv_2crt}
  Let $X=(V,E,\rt,V_{\bullet})$ be a $2$-colored rooted tree and $\bk$ an essentially positive index on $X$. Then we put
  \[\zeta_{\hcS,<M}(X;\bk)\coloneqq\sum_{u\in V_{\bullet}}\sum_{\substack{(m_{v})_{v}\in\bbZ^{V_{\bullet}}\\ m_{v}>0\,(v\neq u)\\ -M<m_{u}<0\\\sum_{v\in V_{\bullet}}m_{v}=0}}\prod_{e\in E}\left(\sum_{e\in P(\rt,v)}(m_{v}+\delta_{u,v}t)\right)^{-\bk(e)}\in\bbQ\jump{t},\]
  for a positive integer $M$, and define the \emph{$t$-adic symmetric multiple zeta value} ($t$-adic SMZV for short) \emph{associated with $X$} as
  \[\zeta_{\hcS}(X;\bk)\coloneqq\lim_{M\to\infty}\zeta_{\hcS,<M}(X;\bk).\]
  Here the limit is taken coefficientwise in $\bbR\jump{t}$. In particular, when $X$ is linear and $(X;\bk)$ is a harvestable pair, there exist positive integers $k_{1},\ldots,k_{r}$ (see the last assertion in Definition \ref{def_2crt}) such that the associated value is written as
  \[\zeta_{\hcS}(X;\bk)=\lim_{M\to\infty}\sum_{i=0}^{r}\sum_{\substack{0<n_{1}<\cdots<n_{i}\\ n_{i+1}<\cdots<n_{r}<0\\ n_{i}-n_{i+1}<M}}\frac{1}{n_{1}^{k_{1}}\cdots n_{i}^{k_{i}}(n_{i+1}+t)^{k_{i+1}}\cdots (n_{r}+t)^{k_{r}}}.\]
  Ono--Seki--Yamamoto \cite{osy21} proved that the right-hand side converges to the value denoted by $\zeta_{\hcS}^{\sh}(k_{1},\ldots,k_{r})$, defined in \cite{hmo21}. Note that the well-definedness of $\zeta_{\hcS}(X;\bk)$ follows from this convergence, Propositions \ref{prop_2_word} and \ref{prop_2_algorithm}.
\end{definition}
The finite and $t$-adic symmetric multiple zeta values associated with $2$-colored rooted trees have the following properties. The shuffle relations for the usual finite and $t$-adic symmetric multiple zeta values are corollaries of them. In the rest of this section, let $\cF$ be either $\cA$ or $\hcS$.
\begin{proposition}[{\cite[Proposition 2.2]{ono17}}, {\cite[Proposition 3.4]{osy21}}]
  Let $X=(V,E,\rt,V_{\bullet})$ be a $2$-colored rooted tree and $\bk$ an index on $X$. Assume that there exists an edge $e=\{v,w\}$ such that $v\in V_{\circ}\setminus\{\rt\}$ and $\bk(e)=0$. Extend $\bk$ as $\bk(\{w,w'\})\coloneqq\bk(\{v,w'\})$ for an adjacent vertex $v$ of $w'$ $(\neq w)$ and 
  \[X'\coloneqq\left(V\setminus\{v\},(E\setminus\{\{v,w'\}\mid w'\in N_{v}\})\cup\{\{w,w'\}\mid w'\in N_{v}\setminus\{w\}\},\rt,V_{\bullet}\right).\]
  This situation is represented as
  \begin{align}
    (X;\bk)=\begin{xy}
      {(0,0)*{\circ} \ar @{-}^{0} (20,0)*{\times}},
      {(-4.5,0)*++[o][F]{T_{1}}},
      {(24,0)*++[o][F]{T_{2}}}
    \end{xy}\qquad\qquad
    (X';\bk)=\begin{xy}
      {(0,0)*{\times}},
      {(-4,0)*++[o][F]{T_{1}}},
      {(4,0)*++[o][F]{T_{2}}}
    \end{xy}
  \end{align}
  Here $\times$ stands for either $\bullet$ or $\circ$. Then we have
  \[\zeta_{\cF}(X;\bk)=\zeta_{\cF}(X';\bk).\]
\end{proposition}
\begin{proposition}[{\cite[Proposition 2.3]{ono17}}, {\cite[Proposition 3.5]{osy21}}]
  Let $X=(V,E,\rt,V_{\bullet})$ be a $2$-colored rooted tree and $\bk$ an index on $X$. Assume that there exists edges $e_{i}=\{v,w_{i}\}$ $(i=1,2)$ such that $v\in V_{\circ}\setminus\{\rt\}$ and $\deg(v)=2$. Moreover, let us write $e\coloneqq\{w_{1},w_{2}\}$,
  \[X'\coloneqq (V\setminus\{v\},(E\setminus\{e_{1},e_{2}\})\cup\{e\},\rt,V_{\bullet})\]
  and extend the domain of $\bk$ to $E\cup\{e\}$ by putting $\bk(e)\coloneqq\bk(e_{1})+\bk(e_{2})$ as follows:
  \[(X;\bk)=\begin{xy}
    {(0,0)*{\circ} \ar @{-}^{\bk(e_{2})} (10,0)*{\times}},
    {(0,0)*{\circ} \ar @{-}_{\bk(e_{1})} (-10,0)*{\times}},
    {(14,0)*++[o][F]{T_{2}}},
    {(-14,0)*++[o][F]{T_{1}}}
  \end{xy}\qquad\qquad
  (X';\bk)=\begin{xy}
    {(-10,0)*{\times} \ar @{-}^{\bk(e_{1})+\bk(e_{2})} (10,0)*{\times}},
    {(14,0)*++[o][F]{T_{2}}},
    {(-14,0)*++[o][F]{T_{1}}}
  \end{xy}\]
  Then we have
  \[\zeta_{\cF}(X;\bk)=\zeta_{\cF}(X';\bk).\]
\end{proposition}
\begin{proposition}[{\cite[Proposition 2.4]{ono17}}, {\cite[Proposition 3.6]{osy21}}]
  Let $X=(V,E,\rt,V_{\bullet})$ be a $2$-colored rooted tree and $\bk$ an index on $X$. Choose an arbitrary vertex $\rt'\in V$ and write $X'\coloneqq (V,E,\rt',V_{\bullet})$. Then we have
  \[\zeta_{\cA}(X;\bk)=(-1)^{\sum_{e\in P(\rt,\rt')}\bk(e)}\zeta_{\cA}(X';\bk)\]
  and
  \[\zeta_{\hcS}(X;\bk)=(-1)^{\sum_{e\in P(\rt,\rt')}\bk(e)}\sum_{\bf\colon P(\rt,\rt')\to\bbZ_{\ge 0}}\left(\sum_{e\in P(\rt,\rt')}\binom{\bk(e)+\bf(e)-1}{\bf(e)}\right)\zeta_{\hcS}(X';\bk\oplus\bf)t^{\sum_{e\in P(\rt,\rt')}\bf(e)}.\]
  Here $\bk\oplus\bf$ denotes the index on $X$ determined by $e\mapsto\bk(e)+\bf(e)$.
\end{proposition}
\begin{proposition}[{\cite[Proposition 3.2]{ono17}}, {\cite[Proposition 3.17]{osy21}}]\label{prop_2_word}
  For a harvestable pair $(X;\bk)$, we have $\zeta_{\cF}(X;\bk)=Z_{\cF}(O_{(X;\bk)})$. Here $Z_{\cF}$ is the unique $\bbQ$-linear map sending $e_{1}e_{0}^{k_{1}-1}\cdots e_{1}e_{0}^{k_{r}-1}$ to $\zeta_{\cF}(k_{1},\ldots,k_{r})$ for $k_{1},\ldots,k_{r}\in\bbZ_{\ge 1}$. 
\end{proposition}
\begin{proposition}[{\cite[Proposition 2.9]{ono17}}, {\cite[Proposition 3.18]{osy21}}]\label{prop_2_algorithm}
  Let $X=(V,E,\rt,V_{\bullet})$ be a $2$-colored rooted tree and $\bk$ an essentially positive index on $X$. Assume $\rt\in V_{\bullet}$. Then there exists a harvestable pair $(X_{\h};\bk_{\h})$ such that
  \[\zeta_{\cF}(X;\bk)=\zeta_{\cF}(X_{\h};\bk_{\h}).\]
\end{proposition}
\section{Iterated integrals}\label{sec_ii}
A \emph{tangential base point} in $\bbC$ is a pair $\tp\coloneqq (p,v)$ consisting of $p\in\bbC$ and its non-zero tangent vector $v$. The first entry is usually identified with the non-zero complex number. For a subset $M\subseteq\bbC$ and tangential base points $\tp=(p,v)$ and $\tq=(q,w)$, a \emph{path} from $\tp$ to $\tq$ on $M$ is a piecewise smooth map $\gamma\colon[0,1]\to\bbC$ such that
\[\gamma(0)=p,\qquad\gamma(1)=q,\qquad\lim_{t\searrow 0}\frac{\gamma(t)-p}{t}=v,\qquad\lim_{t\nearrow 1}\frac{\gamma(t)-q}{t-1}=-w\]
and $\gamma((0,1))\subseteq M$. Homotopy classes of such paths admit the composition and the inverse \cite[\S3.7]{gf22}. Note that, even if we omit the informations of endpoints (and their tangential vectors) as ``a path $\gamma$ on $M$,'' $\gamma(0)$ or $\gamma(1)$ can be in $M$.
\begin{definition}\label{def_general_shuffle}
  For a set $D$, we write $\fH_{D}\coloneqq\bbQ\langle e_{z}\mid z\in D\rangle$. Define the \emph{shuffle product} $\sh$ on $\fH_{D}$ inductively as
  \begin{gather}
    1\sh w=w\sh 1=w,\\
    we_{z}\sh w'e_{z'}=(we_{z}\sh w')e_{z'}+(w\sh w'e_{z'})e_{z}
  \end{gather} 
  for $w,w'\in\fH_{D}$ and $z,z'\in D$. The restriction of this product provides a $\bbQ$-algebra structure on $\fH_{D'}$ for any $\emp\neq D'\subseteq D$.
\end{definition}
Here and in what follows, $x_{1},\ldots,x_{r}$ stand for indeterminates unless noted otherwise. Moreover, we denote by $\bx$ a tuple of indeterminates $(x_{1},\ldots,x_{r})$ and by $\bx^{-1}$ the tuple $(x_{1}^{-1},\ldots,x_{r}^{-1})$ of their inverse.
\begin{definition}
  Let $a_{1},\ldots,a_{k}$ be elements of $\bbC[\bx,\bx^{-1}]$, $\tp$ and $\tq$ be tangential base points of $\bbC$ and $\gamma$ a path from $\tp$ to $\tq$ on $\bbC\setminus D$. Then, similarly to \cite[Lemma 3.238]{gf22}, for $0<z<1$ and integers $m_{1},\ldots,m_{r}$ one has the unique asymptotic expansion of the form
  \begin{align}\label{eq_asymptotic_expansion}
    \begin{split}
    &\left(\text{the coeffcient of }\int_{z<t_{1}<\cdots<t_{k}<1-z}\prod_{i=1}^{r}\frac{d\gamma(t_{i})}{\gamma(t_{i})-a_{i}}\text{ at }x_{1}^{m_{1}}\cdots x_{r}^{m_{r}}\right)\\
    &=\sum_{i=0}^{n}\sum_{j=0}^{\infty}c_{i,j}(m_{1},\ldots,m_{r})(\log z)^{i}z^{j}\qquad (z\to +0).
    \end{split}
  \end{align}
  Thus we can define the \emph{regularized iterated integral} as the $\bbQ$-linear map $I_{\gamma}(\tp;-;\tq)\colon\fH_{\bbC[\bx,\bx^{-1}]}\to\bbC\jump{\bx,\bx^{-1}}$ determined by $1\mapsto 1$ and
  \[I_{\gamma}(\tp;e_{a_{1}}\cdots e_{a_{k}};\tq)=\sum_{m_{1},\ldots,m_{r}\in\bbZ}c_{0,0}(m_{1},\ldots,m_{r})x_{1}^{m_{1}}\cdots x_{r}^{m_{r}}\in\bbC\jump{\bx,\bx^{-1}}.\]
\end{definition}
Note that when the limit $z\to +0$ of the right-hand side of \eqref{eq_asymptotic_expansion} converges (i.e., $n=0$), the value of $I_{\gamma}$ coincides with the sum of their limits over $m_{1},\ldots,m_{r}\ge 0$ and is independent of the choice of tangential vectors.
\begin{proposition}[{\cite[Theorem 3.242]{gf22}}]\label{prop_ii_properties}
  Let $w=e_{a_{1}}\cdots e_{a_{k}}$ $(a_{1},\ldots,a_{k}\in\bbC[\bx,\bx^{-1}])$ and $w'$ be elements of $\fH_{\bbC[\bx,\bx^{-1}]}$, $\tp$, $\tq$ and $\tr$ tangential base points and $\gamma$ $($resp.~$\tilde{\gamma})$ a path from $\tp$ to $\tq$ $($resp.~from $\tq$ to $\tr)$. Then we have
  \begin{enumerate}[\rm (1)]
    \item\label{shuffle} $I_{\gamma}(\tp;w\sh w';\tq)=I_{\gamma}(\tp;w;\tq)I_{\gamma}(\tp;w';\tq)$.
    \item\label{reversal} $I_{\gamma}(\tp;w;\tq)=(-1)^{k}I_{\gamma^{-1}}(\tq;\overleftarrow{w};\tp)$. Here $\overleftarrow{w}$ stands for $e_{a_{k}}\cdots e_{a_{1}}$.
    \item\label{pathcon} $\displaystyle I_{\tilde{\gamma}\circ\gamma}(\tp;w;\tr)=\sum_{i=0}^{k}I_{\gamma}(\tp;e_{a_{1}}\cdots e_{a_{i}};\tq)I_{\tilde{\gamma}}(\tq;e_{a_{i+1}}\cdots e_{a_{k}};\tr)$.
  \end{enumerate}
\end{proposition}
Denote by $\mu_{N}$ the set of $N$-th roots of unity in $\bbC$.
\begin{definition}[\cite{ak04}, {\cite[Definition 3.1]{tasaka21}}, {\cite[(A.1.1)]{jarossay19}}]\label{def_mzvs}
  Let $\tzero$ and $\tone$ be the tangential base points $(0,1)$ and $(1,-1)$ respectively. Fix a positive integers $N$, a tuple of roots of unity $\etab=(\eta_{1},\ldots,\eta_{r})\in\mu_{N}^{r}$ and a tuple of positive integers $\bk=(k_{1},\ldots,k_{r})\in\bbZ_{\ge 1}^{r}$.
  \begin{enumerate}
    \item We define the (\emph{shuffle-regularized}) \emph{multiple $L$-value} by
    \[L^{\sh}\vv{\etab}{\bk}\coloneqq (-1)^{r}I_{\dch}(\tzero;e_{1/\eta_{1}}e_{0}^{k_{1}-1}\cdots e_{1/\eta_{r}}e_{0}^{k_{r}-1};\tone),\]
    where $\dch$ is the path $t\mapsto t$ from $\tzero$ to $\tone$ on $\bbC\setminus\{0,1\}$. We remark that, if $(k_{r},\eta_{r})\neq (1,1)$ this value has a series expression
    \begin{equation}\label{eq_mlv}
      L^{\sh}\vv{\eta_{1},\ldots,\eta_{r}}{k_{1},\ldots,k_{r}}=\sum_{0<n_{1}<\cdots<n_{r}}\frac{\eta_{1}^{n_{1}}\eta_{2}^{n_{2}-n_{1}}\cdots\eta_{r}^{n_{r}-n_{r-1}}}{n_{1}^{k_{1}}\cdots n_{r}^{k_{r}}}.
    \end{equation}
    The $\bbQ$-vector space spanned by all multiple $L$-values and $1$ is denoted by $\cZ_{N}$ (it is also spanned by the values of the form as the right-hand side of \eqref{eq_mlv}). In particular, $6L^{\sh}\vv{1}{2}=\pi^{2}$ is an element of $\cZ_{N}$.
    \item For $\alpha\in\bbZ/N\bbZ$, we define the \emph{symmetric multiple $L$-value} $L_{\alpha}^{\cS}\vv{\etab}{\bk}$ as the projection on $\cZ_{N}/\pi^{2}\cZ_{N}$ of the value
    \[L_{\alpha}^{\cS,\sh}\vv{\etab}{\bk}\coloneqq\sum_{i=0}^{r}(-1)^{k_{i+1}+\cdots+k_{r}}\eta_{i+1}^{\alpha}L^{\sh}\vv{\eta_{1}/\eta_{i+1},\ldots,\eta_{i}/\eta_{i+1}}{k_{1},\ldots,k_{i}}L^{\sh}\vv{1/\eta_{i+1},\eta_{r}/\eta_{i+1},\ldots,\eta_{i+2}/\eta_{i+1}}{k_{r},\ldots,k_{i+1}}.\]
  \end{enumerate}
\end{definition}
\begin{theorem}[{\cite[Corollary 6]{hirose20}}]\label{thm_rsmzv}
  Let $\bk=(k_{1},\ldots,k_{r})$ be a tuple of positive integers, $c$ the path from $\tone$ to itself which circles counterclockwise around $1$ and $\beta\coloneqq\dch^{-1}\circ c\circ\dch$. We put 
  \[\zeta_{\cRS}(\bk)\coloneqq\frac{(-1)^{r}}{2\pi i}I_{\beta}(\tzero;e_{1}e_{0}^{k_{1}-1}\cdots e_{1}e_{0}^{k_{r}-1}e_{1};\tzero).\]
  Then $\zeta_{\cRS}(\bk)\in\cZ_{1}[2\pi i]$ and it coincides with $\zeta_{\cS}(\bk)\coloneqq L_{0}^{\cS}\vv{1,\ldots,1}{\bk}$ modulo $2\pi i$.
\end{theorem}
Generalizing these constructions (that is, considering as the above objects are the ``linear'' case), we attach iterated integrals to general colored rooted trees.
\begin{definition}
  Fix a positive integer $r$ and a finite subset $D$ of $\bbC[\bx,\bx^{-1}]$. 
  \begin{enumerate}
    \item A \emph{$D$-colored rooted tree} $X$ is a tuple $(V,E,\rt,C)$ such that
  \begin{itemize}
    \item a tree $(V,E)$.
    \item a vertex $\rt\in V$, called the \emph{root}.
    \item a map $C\colon V\to D$.
  \end{itemize}
    \item An \emph{index} $\bk$ on $D$-colored rooted tree $X=(V,E,\rt,C)$ is a map $E\to\bbZ_{\ge 0}$. It is said \emph{essentially positive} if $\sum_{e\in P(v,w)}\bk(e)>0$ for distinct vertices $v$ and $w$ satisfying $C(v),C(w)\neq 0$. We simply call a pair of $D$-colored rooted tree and an index on $X$ a \emph{$D$-colored pair}.
    \item For a map $\gamma\colon[0,1]\to\bbC$, we say that a $D$-colored rooted tree $X=(V,E,\rt,C)$ is \emph{$\gamma$-admissible} if the degree of every element of $C^{-1}(\{\gamma(0)\})$ is greater than $1$ and $C(\rt)\notin\{\gamma(0),\gamma(1)\}$.
    \item A $D$-colored rooted tree $X=(V,E,\rt,C)$ is \emph{linear} if $\deg(V)\subseteq\{1,2\}$, $\deg(\rt)=1$ and $C^{-1}(\{0\})=\emp$. 
  \end{enumerate}
\end{definition}
A $2$-colored rooted tree in the sense of Definition \ref{def_2crt} is nothing but a $\gamma$-admissible $\{0,1\}$-colored rooted tree in this definition if $\gamma(0)=0$ and $\gamma(1)\notin\{0,1\}$ (consider as $V_{\bullet}=C^{-1}(\{1\})$). From this analogy, we use the same symbols $\bk_{v}\coloneqq\bk(\{v,p_{v}\})$ and $\bk_{\rt}\coloneqq 1$ for a $D$-colored rooted tree $(V,E,\rt,C)$ and an index $\bk$ on $X$ as Definition \ref{def_2crt} \eqref{def_2crt_index}.
\begin{remark}
  For finite subsets $D$ and $D'$ of $\bbC[\bx,\bx^{-1}]$, every $D$-colored rooted tree is also a $D\cup D'$-colored rooted tree.
\end{remark}
\begin{definition}\label{def_ii_dcrt}
  Let $D$ be a finite subset of $\bbC[\bx,\bx^{-1}]$, $a_{1},\ldots,a_{k}$ elements of $D$, $\tp$ and $\tq$ tangential base points in $\bbC$, $\gamma$ a path from $\tp$ to $\tq$ on $\bbC\setminus D$, $X=(V,E,\rt,C)$ a $D$-colored rooted tree and $\bk$ an index on $X$. Assume that $X$ is $\gamma$-admissible for the convergence. Then we put
  \[\Delta(X;\bk)\coloneqq\left\{(t_{v})_{v\in V}\in [0,1]^{|V|}~\left|~\begin{array}{cc}t_{v}<t_{p_{v}} & \text{if } \bk_{v}>0,\\ t_{v}=t_{p_{v}} & \text{if } \bk_{v}=0.\end{array}\right.\right\}\]
  and define the (convergent) \emph{iterated integral associated with $X$} by
  \[I_{\gamma}(X;\bk)\coloneqq\int_{(t_{v})_{v}=\bt\in\Delta(X;\bk)}\prod_{v\in V}F_{v,\bk_{v}}(\gamma(\bt))\in\bbC\jump{\bx,\bx^{-1}},\]
  where $\gamma(\bt)$ stands for $(\gamma(t_{v}))_{v\in V}$ and we put
  \[F_{v,k}(\bt)\coloneqq\begin{dcases}\frac{1}{(k-1)!}(\log t_{p_{v}}-\log t_{v})^{k-1}\frac{dt_{v}}{t_{v}-C(v)} & \text{if } k>0,\\ \frac{t_{v}}{t_{v}-C(v)} & \text{if } k=0.\end{dcases}\]
\end{definition}
Note that the integration domain $\Delta(X;\bk)$ is a subset of $[0,1]^{|V|}$, but the actual number of variables for integration is $|\{v\in V\mid\bk_{v}>0\}|$.
Throughout the following two lemmas, we assume that $D$ is a finite subset of $\bbC[\bx,\bx^{-1}]$, $\gamma$ is a path on $\bbC\setminus D$, $X=(V,E,\rt,C)$ is a $\gamma$-admissible $D$-colored rooted tree and $\bk$ is an index on $X$.
\begin{proposition}\label{prop_p_contraction_1}
  Assume that there is an edge $e=\{v,w\}$ such that $C(v)=0$, $v\neq\rt$ and $w\in N_{v}$. Define $\bk(\{w,w'\})\coloneqq\bk(\{v,w'\})$ for a vertex $w'\in N_{v}\setminus\{w\}$ and 
  \[X'\coloneqq\left(V\setminus\{v\},(E\setminus\{\{v,w'\}\mid w'\in N_{v}\})\cup\{\{w,w'\}\mid w'\in N_{v}\setminus\{w\}\},\rt,C|_{V\setminus\{v\}}\right).\]
  Then we have 
  \[I_{\gamma}(X;\bk)=I_{\gamma}(X';\bk).\]
\end{proposition}
\begin{proof}
  Let $\bs=(s_{u})_{u\in V}\in\Delta(X;\bk)$, $\bt=(t_{u})_{u\in V}\coloneqq (\gamma(s_{u}))_{u\in V}$ and write $\bk^{X}_{u}$ (resp.~$\bk^{X'}_{u}$ or $p^{X'}_{u}$) instead of $\bk_{u}$ in $X$ (resp.~$\bk_{u}$ or $p_{u}$ in $X'$) to distinguish them. We can get an element of $\Delta(X';\bk)$ from $\bs$ by removing $s_{v}$ or $s_{w}$ because $s_{v}=s_{w}$ by the assumption. Thus the actual integration domain of $I_{\gamma}(X';\bk)$ agrees with that of $I_{\gamma}(X;\bk)$. Moreover, for $u\in V$ which is not $v$, $w$ nor having $v$ as its parent, we have
  \begin{equation}\label{eq_each_form}
    F_{u,\bk^{X}_{u}}(\bt)=F_{u,\bk^{X'}_{u}}(\bt)
  \end{equation}
  because there is no change about $u$ from $X$ to $X'$. Even $p_{u}=v$, the equality \eqref{eq_each_form} holds by the extended definition of $\bk$. Therefore it suffices to prove
  \[F_{v,\bk^{X}_{v}}(\bt)F_{w,\bk^{X}_{w}}(\bt)=F_{w,\bk^{X'}_{w}}(\bt).\]
  In the case where $w=p_{v}$, the assertion follows from $F_{v,\bk_{v}}(\bt)=1$ and that $\{w,p_{w}\}$ is the same edge in both $X$ and $X'$.
  If $v=p_{w}$ and $\bk^{X}_{v}>0$, then
  \begin{align}
    F_{w,\bk^{X}_{w}}(\bt)F_{v,\bk^{X}_{v}}(\bt)
    &=\frac{t_{w}}{t_{w}-C(w)}\frac{1}{(\bk^{X}_{v}-1)!}(\log t_{p^{X}_{v}}-\log t_{v})^{\bk^{X}_{v}-1}\frac{dt_{v}}{t_{v}}\\
    &=\frac{1}{(\bk^{X'}_{w}-1)!}(\log t_{p^{X'}_{w}}-\log t_{v})^{\bk^{X'}_{w}-1}\frac{dt_{w}}{t_{w}-C(w)}\\
    &=F_{w,\bk^{X'}_{w}}(\bt).
  \end{align}
  If $v=p_{w}$ and $\bk^{X}_{v}=0$, then $\bk^{X'}_{w}=\bk^{X}_{v}=0$ and thus
  \[F_{w,\bk^{X}_{w}}(\bt)F_{v,\bk^{X}_{v}}(\bt)=\frac{t_{w}}{t_{w}-C(w)}=F_{w,\bk^{X'}_{w}}(\bt).\]
\end{proof}
\begin{proposition}\label{prop_p_contraction_2}
  Assume that there are vertices $w_{1}$, $w_{2}$ and $v$ such that $C(v)=0$, $v\neq\rt$, $\deg(v)=2$ and $e_{i}\coloneqq\{w_{i},v\}\in E$ $(i=1,2)$. Put $e\coloneqq\{w_{1},w_{2}\}$ and $\bk(e)\coloneqq\bk(e_{1})+\bk(e_{2})$. Then we have
  \[I_{\gamma}(X;\bk)=I_{\gamma}(V\setminus\{v\},(E\setminus\{e_{1},e_{2}\})\cup\{e\},\rt,C;\bk|_{(E\setminus\{e_{1},e_{2}\})\cup\{e\}}).\]  
\end{proposition}
\begin{proof}
  We prove only in case $\tp=\tzero$, $\tq=\tone$ and $\gamma=\dch$ (other cases are similar). If one of $\bk(e_{1})$ and $\bk(e_{2})$ is zero, the assertion follows from Proposition \ref{prop_p_contraction_1}. Hence we assume that $\bk(e_{i})$ is positive for $i\in\{1,2\}$. Let $\bt\in\Delta(X;\bk)$ and the convention about the symbols $\bk^{X'}_{u}$, $p^{X'}_{u}$, etc.~the same as Proposition \ref{prop_p_contraction_1}. Without loss of generality, we can assume that $p^{X}_{v}=w_{1}$. If a vertex $u$ is neither $v$ nor $w_{2}$, the equality \eqref{eq_each_form} is true. Then it is sufficient to prove the equality between $1$-forms
  \[\int_{t_{w_{2}}<t_{v}<t_{w_{1}}}F_{v,\bk^{X}_{v}}(\bt)F_{w_{2},\bk^{X}_{w_{2}}}(\bt)=F_{w_{2},\bk^{X'}_{w_{2}}}(\bt)\]
  because the range where $t_{v}$ moves is $[t_{w_{2}},t_{w_{1}}]$ for $\bt=(t_{u})_{u\in V}\in\Delta(X;\bk)$. This equality follows from
  \begin{align}
    &\int_{t_{w_{2}}<t_{v}<t_{w_{1}}}F_{w_{2},\bk^{X}_{w_{2}}}(\bt)F_{v,\bk^{X}_{v}}(\bt)\\
    &=\int_{t_{w_{2}}}^{t_{w_{1}}}\frac{1}{(\bk(e_{2})-1)!}(\log t_{v}-\log t_{w_{2}})^{\bk(e_{2})-1}\frac{dt_{w_{2}}}{t_{w_{2}}-C(w_{2})}\frac{1}{(\bk(e_{1})-1)!}(\log t_{w_{1}}-\log t_{v})^{\bk(e_{1})-1}\frac{dt_{v}}{t_{v}}\\
    &=\int_{t_{w_{2}}<x_{1}<\cdots<x_{\bk(e_{1})-1}<t_{v}<y_{1}<\cdots<y_{\bk(e_{1})-1}<t_{w_{1}}}\frac{dx_{1}}{x_{1}}\cdots\frac{dx_{\bk(e_{1})-1}}{x_{\bk(e_{1})-1}}\frac{dt_{v}}{t_{v}}\frac{dy_{1}}{y_{1}}\cdots\frac{dy_{\bk(e_{2})-1}}{y_{\bk(e_{2})-1}}\cdot\frac{dt_{w_{2}}}{t_{w_{2}}-C(w_{2})}\\
    &=\frac{1}{(\bk(e_{1})+\bk(e_{2})-1)!}(\log t_{w_{1}}-\log t_{w_{2}})^{\bk(e_{1})+\bk(e_{2})-1}\frac{dt_{w_{2}}}{t_{w_{2}}-C(w_{2})}\\
    &=F_{w_{2},\bk^{X'}_{w_{2}}}(\bt).
  \end{align}
  We used the extended definition $\bk^{X'}_{w_{2}}=\bk(e_{1})+\bk(e_{2})$ in the last equality.
\end{proof}
\begin{definition}
  Let $D$ be a finite subset of $\bbC[\bx,\bx^{-1}]$, $\gamma$ a path on $\bbC\setminus D$, $X=(V,E,\rt,C)$ a $D$-colored rooted tree and $\bk$ an index on $X$. We say that the pair $(X;\bk)$ is \emph{$\gamma$-harvestable} if all of the following conditions are satisfied (replace $D$ by $D\cup\{0\}$ if necessary):
  \begin{itemize}
    \item $X$ is $\gamma$-admissible.
    \item $\deg(\rt)=1$ (and thus $C(\rt)\neq\gamma(0)$).
    \item $\deg(v)\le 2$ (resp.~$\neq 2$) for any $v\in V$ such that $C(v)\neq 0$ (resp.~$C(v)=0$).
    \item For $v\in V$, if $\bk_{v}=0$ then $C(v)=0$ and $C(p_{v})\neq 0$ (and thus $\bk$ is essentially positive).
  \end{itemize}
\end{definition}
When $D=\{0,1\}$, $\gamma(0)=0$ and $\gamma(1)\notin\{0,1\}$, the above definition amounts to Definition \ref{def_2crt} \eqref{def_harvestability}.
\begin{example}\label{ex_linear}
  Let us consider the iterated integrals associated with a $\gamma$-harvestable $D$-colored rooted tree $(X;\bk)$ in the case where $X$ is linear. Then according to the last assertion of Definition \ref{def_2crt}, the pair $(X;\bk)=(V,E,\rt,C;\bk)$ is uniquely written as
  \begin{align}
    V&=\{v_{1},\ldots,v_{r+1}\},\\
    E&=\{\{v_{i},v_{i+1}\}\mid i=1,\ldots,r\},\\
    \rt&=v_{r+1},
  \end{align}
  and $\bk(\{v_{i},v_{i+1}\})=k_{i}$ ($i=1,\ldots,r$) by $k_{1},\ldots,k_{r}>0$. Moreover, when $D=\{0,1\}$, the image of $C$ is $\{1\}$. Then we can write down the associated iterated integral (along the path $\gamma$ between tangential base points $\tp$ and $\tq$) as
  \begin{align}
    I_{\gamma}(X;\bk)
    &=\int_{0<t_{1}<\cdots<t_{r+1}<1}\prod_{i=1}^{r+1}F_{v_{i},\bk(\{v_{i},v_{i+1}\})}(\gamma(t_{1}),\ldots,\gamma(t_{r+1}))\\
    &=\int_{0<t_{1}<\cdots<t_{r+1}<1}\left(\prod_{i=1}^{r}\frac{1}{(k_{i}-1)!}(\log \gamma(t_{i+1})-\log \gamma(t_{i}))^{k_{i}-1}\frac{d\gamma(t_{i})}{\gamma(t_{i})-1}\right)\cdot\frac{d\gamma(t_{r+1})}{\gamma(t_{r+1})-1}\\
    &=I_{\gamma}(\tp;e_{1}e_{0}^{k_{1}-1}\cdots e_{1}e_{0}^{k_{r}-1}e_{1};\tq),
  \end{align}
  by the usual iterated integral symbol. According to Theorem \ref{thm_rsmzv}, putting $\gamma=\beta$ we have
  \[I_{\beta}(X;\bk)=(-1)^{r}2\pi i\zeta_{\cRS}(k_{1},\ldots,k_{r})\]
  and thus our integrals include symmetric multiple zeta values. Moreover, when $\gamma$ is the path $\dch_{0,z}$ defined by $t\mapsto zt$ from $\tzero$ to the tangential base point $(z,-z)$ for $0<z<1$, we obtain
  \begin{equation}\label{eq_polylog}
    \frac{d}{dz}I_{\dch_{0,z}}(X;\bk)=(-1)^{r+1}\sum_{M=1}^{\infty}\left(\sum_{0<n_{1}<\cdots<n_{r}\le M}\frac{1}{n_{1}^{k_{1}}\cdots n_{r}^{k_{r}}}\right)z^{M}
  \end{equation}
  by expanding power series and thus our integral also include finite multiple zeta values.
\end{example}
\begin{lemma}\label{p_harvestability_characterization}
  Let $D$ be a finite subset of $\bbC[\bx,\bx^{-1}]$, $\gamma$ a path on $\bbC\setminus D$, $(X;\bk)=(V,E,\rt,C;\bk)$ a $\gamma$-harvestable $D$-colored pair. Then there exists a finite subset $D'$ of $\bbC[\bx,\bx^{-1}]$ and a finite non-empty set of $D'$-colored pairs $\{(X_{i};\bk^{(i)})\}_{i}=\{(V_{i},E_{i},\rt_{i},C_{i};\bk^{(i)})\}_{i=0,\ldots,r}$ such that all of the following conditions are satisfied:
  \begin{itemize}
  \item $(X_{1},\bk^{(1)}),\ldots,(X_{r},\bk^{(r)})$ are $\gamma$-harvestable.
  \item $C_{i}(V_{i}\setminus\{\rt_{i}\})\subseteq D$.
  \item $X_{0}$ is linear.
  \item $\bk^{(0)}(e)>0$ for $e\in E_{0}\setminus\{\{\rt_{0},u_{0}\}\}$ $(u_{0}\in V_{0}\setminus\{\rt_{0}\}$ is nearest to $\rt_{0})$.
  \item $\rt$ is the farthest vertex from $\rt_{0}$ of $X_{0}$.
  \item $(V,E)=G(X_{0},\ldots,X_{r})$ and the image of the cluster vertex under $C$ is $0$.
  \end{itemize}
  Moreover, the set $\{(X_{i};\bk^{(i)})\}_{i}$ is unique up to the choice of $C(\rt_{0}),\ldots,C(\rt_{r})$.
\end{lemma}
\begin{proof}
  When $C^{-1}(\{0\})=\emp$, the lemma is true with $r=0$ because $X$ is linear by the $\gamma$-harvestability.
  Let $(X;\bk)=(V,E,\rt,C;\bk)$ be a $\gamma$-harvestable $D$-colored pair and $v'$ the branched vertex nearest to $\rt$. Then we define $V_{0}\coloneqq\{u\in V\mid P(u,\rt)\subseteq P(v',\rt)\}$, let $(V_{0},E_{0})$ be the subgraph of $X$ induced by $V_{0}$, put $\rt_{0}=v$ and
  \[C_{0}(u)\coloneqq\begin{cases}C(u) & \text{if }u\in V_{0}\setminus\{v'\},\\ \text{an element of }D\setminus\{0\} & \text{if }u=v'.\end{cases}\]
  These data gives $X_{0}$ appearing in the assertion. We can also give $\bk^{(0)}\coloneqq\bk|_{E_{0}}$. Next, put $\{v_{1},\ldots,v_{r}\}\coloneqq\{u\in V\mid \{u,v'\}\in P(u,\rt)\}$, prepare a new vertex $w_{i}$ for each $1\le i\le r$,
  \begin{align}
    V_{i}&\coloneqq\{u\in V\mid P(v_{i},\rt)\subseteq P(u,\rt)\}\cup\{w_{i}\},\\
    E_{i}&\coloneqq\{e\in E\mid \text{both endpoints}\in V_{i}\}\cup\{\{v_{i},w_{i}\}\},\\
    \rt_{i}&\coloneqq w_{i},\\
    C_{i}(u)&\coloneqq\begin{cases}C(u) & \text{if }u\in V_{i}\setminus\{w_{i}\},\\ \text{an element of }\bbC[\bx,\bx^{-1}]\setminus\{\gamma(0),\gamma(1)\} & \text{if }u=w_{i}.\end{cases}
  \end{align}
  and $\bk^{(i)}\coloneqq\bk|_{E_{i}}$. By putting $D'\coloneqq D\cup\{C(\rt_{1}),\ldots,C(\rt_{r})\}$, these satisfy the required conditions.
\end{proof}
\begin{definition}\label{word_with_p_tree}
  Let $D$ be a finite subset of $\bbC[\bx,\bx^{-1}]$, $\gamma$ a path on $\bbC\setminus D$ and $(X;\bk)$ a $\gamma$-harvestable pair with a $D$-colored rooted tree $X$. We assign an element $w_{(X;\bk)}$ of $\fH_{D}$ to the pair $(X;\bk)$ by the following recursive rules: let $B_{P}$ denote the set of vertices whose degree is greater than $2$ in $D$-colored pair $P$. If $B_{(X;\bk)}=0$, the pair $(X;\bk)$ is uniquely written as
  \begin{align}
    V&=\{v_{1},\ldots,v_{r+1}\},\\
    E&=\{\{v_{i},v_{i+1}\}\mid i=1,\ldots,r\},\\
    \rt&=v_{r+1},\\
    C(v_{i})&=z_{i}\qquad (i=1,\ldots,r+1),
  \end{align}
  and $\bk(\{v_{i},v_{i+1}\})=k_{i}$ ($i=1,\ldots,r$) by using elements $z_{1},\ldots,z_{r+1}$ of $D$ and positive integers $k_{1},\ldots,k_{r}$. Then we define
  \[w_{(X;\bk)}\coloneqq e_{z_{1}}e_{0}^{k_{1}-1}\cdots e_{z_{r}}e_{0}^{k_{r}-1}e_{z_{r+1}}.\]
  Next, assume that there exists a positive integer $k$ such that $w_{P}$ is already defined for any $\gamma$-harvestable $D$-colored pair $P$ satisfying $B_{P}<k$. When $(X;\bk)=(V,E,\rt,C;\bk)$ is a $\gamma$-harvestable $D$-colored pair such that $X$ is not linear and $B_{(X;\bk)}=k$, we put
  \[w_{(X;\bk)}\coloneqq(\tilde{w}_{(X_{1};\bk^{(1)})}\sh\cdots\sh\tilde{w}_{(X_{r};\bk^{(r)})})e_{0}^{\bk^{(0)}(\{\rt_{0},v'\})}w_{(V_{0}\setminus\{\rt_{0}\},E_{0}\setminus\{\{\rt_{0},u_{0}\}\},\rt,C_{0}|_{V_{0}\setminus\{\rt_{0}\}};\bk^{(0)}|_{E_{0}\setminus\{\{\rt_{0},v'\}\}})},\]
  with the same symbols as Lemma \ref{p_harvestability_characterization}. Here $\tilde{w}$ denotes the image of $w$ under the $\bbQ$-linear map on $\fH_{\bbC[\bx,\bx^{-1}]}$ determined by $1\mapsto 0$ and $w'e_{z}\mapsto w'$ ($w'\in\fH_{\bbC[\bx,\bx^{-1}]}$ and $z\in\bbC[\bx,\bx^{-1}]$). We remark that this definition is independent of the choice of $C(\rt_{0}),\ldots,C(\rt_{r})$, which is the non-unique part of Lemma \ref{p_harvestability_characterization}.
\end{definition}
\begin{remark}
  Even if $p=0$ and $D=\{0,1\}$, the element $w_{(X;\bk)}$ in Definition \ref{word_with_p_tree} is not equal to $O_{(X';\bk)}$ in the sense of Definition \ref{word_with_tree} associated with the $2$-colored rooted tree $X'$ corresponding to $X$. The latter $O_{(X';\bk)}$ always coincides with the rest $\tilde{w}_{(X;\bk)}$ after removing the last letter of $w_{(X;\bk)}$.
\end{remark}
\begin{theorem}\label{thm_shuffle_theorem}
  Let $D$ be a finite subset of $\bbC$, $\tp$ and $\tq$ tangential base points in $\bbC$, $\gamma$ a path from $\tp$ to $\tq$ in $\bbC\setminus D$, $(X;\bk)=(V,E,\rt,C;\bk)$ be a $\gamma$-harvestable $D$-colored rooted tree and $\bk$ an index on $X=(V,E,\rt,C)$. Then we have
  \begin{equation}\label{eq_shuffle_theorem}
    I_{\gamma}(X;\bk)=I_{\gamma}(\tp;w_{(X;\bk)};\tq).
  \end{equation}
\end{theorem}
\begin{proof}
We prove this proposition by induction on $\ell(X;\bk)\coloneqq\sum_{e\in E\setminus E_{0}}\bk(e)$ (where $E_{0}$ means the set of edges of $X$ contained in $P(v',\rt)$ and $v'$ means the branched vertex nearest to the root). Example \ref{ex_linear} includes the case where $\ell(X;\bk)=1$. Assume the assertion for any $\gamma$-harvestable $D$-colored pair $P$ satisyfing $\ell(P)=\ell-1$ for some $\ell>0$. Take a $\gamma$-harvestable $D$-colored pair $(X;\bk)$ with $\ell(X;\bk)=\ell$ and $\{(X_{i};\bk^{(i)})\}_{i}=\{(V_{i},E_{i},\rt_{i},C_{i};\bk^{(i)})\}_{i=0,\ldots,r}$ as in Lemma \ref{p_harvestability_characterization} (then their grafting is $(X;\bk)$ and the cluster vertex is $v'$). Denote by $v'_{i}$ the vertex such that $\{v'_{i},\rt_{i}\}\in E_{i}$ for $i=0,\ldots,r$. This is unique by the harvestability ($i=1,\ldots,r$) or linearity ($i=0$). Then $k'_{i}\coloneqq\bk(\{v',v'_{i}\})$ is positive when $i\neq 0$ and thus we can define a new index $\bk'_{j}$ on $X$ for each $j=1,\ldots,r$ as
\[\bk'_{j}(e)\coloneqq\begin{cases}k'_{j}-1 & \text{if } e=\{v',v'_{j}\},\\ \bk(e)+1 & \text{if } e=\{v',p^{X}_{v'}\},\\ \bk(e) & \text{otherwise}.\end{cases}\]
Then we have
\begin{equation}\label{eq_osy_37}
  I_{\gamma}(X;\bk)=\sum_{j=1}^{r}I_{\gamma}(X;\bk'_{j})
\end{equation}
as \cite[Lemma 3.7]{osy21} by shuffling the domain of integration. Indeed, we can calculate the left-hand side as
\begin{align}
  &I_{\gamma}(X;\bk)\\
  &=\int_{\bt=(t_{v})_{v\in V}\in\Delta(X;\bk)}\prod_{v\in V}F_{v,\bk_{v}}(\gamma(\bt))\\
  &=\int_{\bt}\left(\prod_{v\in V\setminus\{v'_{1},\ldots,v'_{r}\}}F_{v,\bk_{v}}(\gamma(\bt))\right)\left(\prod_{i=1}^{r}\frac{d\gamma(t_{v'_{i}})}{\gamma(t_{v'_{i}})-C(v'_{i})}\int_{t_{v'_{i}}<u_{1}<\cdots<u_{k'_{i}-1}<t_{v'}}\frac{d\gamma(u_{1})}{\gamma(u_{1})}\cdots\frac{d\gamma(u_{k'_{i}-1})}{\gamma(u_{k'_{i}-1})}\right)\\
  &=\int_{\bt}\left(\prod_{v\in V\setminus\{v'_{1},\ldots,v'_{r}\}}F_{v,\bk_{v}}(\gamma(\bt))\right)\\
  &\qquad\cdot\left(\sum_{j=1}^{r}\prod_{i=1}^{r}\frac{d\gamma(t_{v'_{i}})}{\gamma(t_{v'_{i}})-C(v'_{i})}\int_{t_{v'_{i}}<u_{1}<\cdots<u_{k'_{i}-1-\delta_{i,j}}<u_{k'_{j}-1}<t_{v'}}\frac{d\gamma(u_{1})}{\gamma(u_{1})}\cdots\frac{d\gamma(u_{k'_{i}-1-\delta_{i,j}})}{\gamma(u_{k'_{i}-1-\delta_{i,j}})}\frac{d\gamma(u_{k'_{j}-1})}{\gamma(u_{k'_{j}-1})}\right)\\
  &=\int_{\bt}\left(\prod_{v\in V\setminus\{v',v'_{1},\ldots,v'_{r}\}}F_{v,\bk_{v}}(\gamma(\bt))\right)\cdot F_{v',\bk_{v'}+1}(\gamma(\bt))\cdot\left(\sum_{j=1}^{r}\prod_{i=1}^{r}F_{v'_{i},k'_{i}-\delta_{i,j}}(\bt)\right)\\
  &=\sum_{j=1}^{r}I_{\gamma}(X;\bk'_{j}).
\end{align}
Here we considered $u_{k'_{i}-1}<u_{k'_{j}-1}$ for $i\neq j$ in the third equality and interchanged names of $u_{k'_{i}-1}$ and $t_{v'}$ in the fourth equality. This calculation remains true even $k'_{j}=1$ by considering $u_{k'_{i}-1}<t_{v'_{j}}$ instead of $u_{k'_{i}-1}<u_{k'_{j}-1}$ in the third equality. When $k'_{j}>1$, since the pair $(X;\bk'_{j})$ is $\gamma$-harvestable and satisfies $\ell(X;\bk'_{j})=\ell-1$, we have
\begin{align}\label{eq_partial_word_1}
  \begin{split}
  w_{(X;\bk'_{j})}
  &=(w_{(X_{1};\bk^{(1)})}\sh\cdots\sh R_{0}^{-1}(\tilde{w}_{(X_{j};\bk^{(j)})})\sh\cdots\sh w_{(X_{r};\bk^{(r)})})e_{0}^{\bk^{(0)}(\{\rt_{0},v'\})+1}\\
  &\qquad\cdot w_{(V_{0}\setminus\{\rt_{0}\},E_{0}\setminus\{\{\rt_{0},u_{0}\}\},\rt,C_{0}|_{V_{0}\setminus\{\rt_{0}\}};\bk^{(0)}|_{E_{0}\setminus\{\{\rt_{0},v'\}\}})}
  \end{split}
\end{align}
by the induction hypothesis ($R_{z}^{-1}\colon\bbQ\oplus\fH_{\bbC[\bx,\bx^{-1}]}e_{z}\to\fH_{\bbC[\bx,\bx^{-1}]}$ is the $\bbQ$-linear map determined by $1\mapsto 0$ and $we_{z}\mapsto w$ for $w\in\fH_{\bbC[\bx,\bx^{-1}]}$). When $k'_{j}=1$, using Proposition \ref{prop_p_contraction_1} twice, we see that $I_{\gamma}(X;\bk'_{j})$ is equal to $I_{\gamma}(X';\bk''_{j})$ with $X'=(V,E'_{j},\rt,C)$,
\[E'_{j}\coloneqq(E\setminus(\{v',u_{0}\}\cup\{\{v'_{j},w\}\mid w\in A_{v'_{j}}\setminus\{v'\}\}))\cup\{\{v',w\}\mid w\in A_{v'_{j}}\setminus\{v'\}\}\cup\{\{v'_{j},u_{0}\}\}\]
and
\[\bk''_{j}(e)\coloneqq\begin{cases}\bk(\{v'_{j},w'\}) & \text{if } e=\{v',w\},~w\in N_{v'_{j}}\setminus\{v'\},\\ \bk(\{v',u_{0}\}) & \text{if } e=\{v'_{j},u_{0}\},\\ \bk(e) & \text{otherwise}.\end{cases}\]
Then it follows that
\begin{align}\label{eq_partial_word_2}
  \begin{split}
  w_{(X';\bk''_{j})}
  &=(w_{(X_{1};\bk^{(1)})}\sh\cdots\sh R_{C_{j}(v'_{j})}^{-1}(\tilde{w}_{(X_{j};\bk^{(j)})})\sh\cdots\sh w_{(X_{r};\bk^{(r)})})e_{C(v'_{j})}e_{0}^{\bk^{(0)}(\{\rt_{0},v'\})}\\
  &\qquad\cdot w_{(V_{0}\setminus\{\rt_{0}\},E_{0}\setminus\{\{\rt_{0},u_{0}\}\},\rt,C_{0}|_{V_{0}\setminus\{\rt_{0}\}};\bk^{(0)}|_{E_{0}\setminus\{\{\rt_{0},v'\}\}})}.
  \end{split}
\end{align}
Combining \eqref{eq_partial_word_1}, \eqref{eq_partial_word_2} and Definition \ref{def_general_shuffle} shows that
\begin{equation}\label{eq_inductive_word}
  \sum_{\substack{1\le j\le r\\ k'_{j}\ge 2}}w_{(X;\bk'_{j})}+\sum_{\substack{1\le j\le r\\ k'_{j}=1}}w_{(X';\bk''_{j})}=w_{(X;\bk)}.
\end{equation}
Furthermore, since if $k'_{j}\ge 2$ (resp.~$k'_{j}=1$) then $\ell(X;\bk'_{j})=\ell-1$ (resp.~$\ell(X';\bk''_{j})=\ell-1$), using \eqref{eq_osy_37}, \eqref{eq_inductive_word} and the induction hypothesis we have
\begin{align}
  I_{\gamma}(X;\bk)
  &=\sum_{j=1}^{r}I_{\gamma}(X;\bk'_{j})\\
  &=\sum_{\substack{1\le j\le r\\ k_{j}\ge 2}}I_{\gamma}(X;\bk'_{j})+\sum_{\substack{1\le j\le r\\ k_{j}=1}}I_{\gamma}(X';\bk''_{j})\\
  &=I_{\gamma}\left(\tp;\sum_{\substack{1\le j\le r\\ k_{j}\ge 2}}w_{(X;\bk'_{j})}+\sum_{\substack{1\le j\le r\\ k_{j}=1}}w_{(X';\bk''_{j})};\tq\right)\\
  &=I_{\gamma}(\tp;w_{(X;\bk)};\tq).
\end{align}
\end{proof}
\begin{proposition}\label{prop_p_representation_algorithm}
  Let $D$ be a finite subset of $\bbC[\bx,\bx^{-1}]$, $\gamma$ a path on $\bbC\setminus D$, $X=(V,E,\rt,C)$ be a $\gamma$-admissible $D$-colored rooted tree with $C(\rt)\notin\gamma(0)$ and $\bk$ an essentially positive index on $X$. Then there exists a $\gamma$-harvestable $D$-colored pair $(X_{\har};\bk_{\har})$ such that
  \[I_{\gamma}(X;\bk)=I_{\gamma}(X_{\har};\bk_{\har}).\]
\end{proposition}
\begin{proof}
  The assertion can be proved in the same way as \cite[Proposition 3.18]{osy21}. Indeed, $(X_{\har};\bk_{\har})$ is obtained in the following processes:
  \begin{enumerate}
    \item If $(X;\bk)$ has an edge $e=\{v,w\}$ such that $C(v)=0$ and $\bk(e)=0$, use Proposition \ref{prop_p_contraction_1} repeatedly to delete such $e$.
    \item If $(X;\bk)$ has a vertex $v$ such that $C(v)=0$ and $\deg(v)=2$, use Proposition \ref{prop_p_contraction_2} repeatedly to delete such $v$.
    \item If $(X;\bk)$ has a vertex $v$ such that $C(v)\neq 0$ and $\deg(v)\ge 3$, use Proposition \ref{prop_p_contraction_1} repeatedly to delete such $v$.
    \item If $\deg(\rt)\ge 2$, use Proposition \ref{prop_p_contraction_1} to make the root a terminal.
  \end{enumerate}
  The $\gamma$-harvestability of $(X_{\har};\bk_{\har})$ is also proved similarly to \cite[Proposition 3.18]{osy21}.
\end{proof}
\section{Regularization and applications for finite and symmetric multiple polylogarithms}\label{sec_mlvs}
In this section, we extend Definition \ref{def_ii_dcrt} to general $D$-colored rooted trees (with essentially positive indices) by the method called \emph{regularization} and apply it to prove the $\bp$-adic and (resp.~$t$-adic) shuffle relations for finite (resp.~symmetric) multiple polylogarithms.
\subsection{Regularization and changing roots}
\begin{definition}
  Let $D$ be a finite subset of $\bbC[\bx,\bx^{-1}]$, $\gamma$ a path on $\bbC\setminus D$, $X=(V,E,\rt,C)$ a $D$-colored rooted tree and $\bk$ an essentially positive index on $X$. For $0<z<1$ and $0\le t\le 1$, define $\gamma_{z}(t)\coloneqq\gamma(z+(1-2z)t)$. Then there exists $z'\in (0,1/2)$ such that $\gamma(t)\notin D$ and $\gamma(u)\neq C(\rt)$ for every $t\in (0,z')$ and $u\in (0,z')\cup (1-z',1)$. Thus $X$ is $\gamma_{z}$-admissible when $z\in (0,z')$ and we obtain the $\gamma_{z}$-harvestable $D$-colored pair $(X_{\har};\bk_{\har})$ from Proposition \ref{prop_p_representation_algorithm}. By Theorem \ref{thm_shuffle_theorem} and \eqref{eq_asymptotic_expansion}, we get the asymptotic expression
  \begin{align}
    I_{\gamma_{z}}(X;\bk)
    &=I_{\gamma_{z}}(X_{\har};\bk_{\har})\\
    &=I_{\gamma_{z}}((\gamma(z),(1-2z)\gamma'(z));w_{(X_{\har};\bk_{\har})};(\gamma(1-z),(2z-1)\gamma'(1-z)))\\
    &=\sum_{i=0}^{N}\sum_{j=0}^{\infty}c_{i,j}(\log z)^{i}z^{j}\qquad (z\to +0),
  \end{align}
  where $c_{i,j}\in\bbC\jump{\bx,\bx^{-1}}$ and we define the \emph{regularized iterated integral associated with $(X;\bk)$} as
  \[I_{\gamma}(X;\bk)\coloneqq c_{0,0}.\] 
\end{definition}
When $(X;\bk)$ is $\gamma$-admissible, the limit $\lim_{z\to 0}I_{\gamma_{z}}(X;\bk)$ converges coefficientwise as the multivariable Laurent series of $x_{1},\ldots,x_{r}$ and this definition agrees with Definition \ref{def_ii_dcrt}.
\begin{proposition}\label{prop_root_change}
  Let $D$ be a finite subset of $\bbC[\bx,\bx^{-1}]$, $\gamma$ a path on $\bbC\setminus D$, $X=(V,E,\rt,C)$ a $D$-colored rooted tree and $\bk$ an essentially positive index on $X$. Choose a vertex $\rt'\in V$ and write $X'\coloneqq (V,E,\rt',C)$. Then there exist two families $\{(Y_{j};\bl^{(j)})\}_{j=1,\ldots,N}$ and $\{(Z_{j};\bh^{(j)})\}_{j=1,\ldots,N}$ of $D$-colored pairs and a sequence $\{a_{j}\}_{j=1,\ldots,N}$ of integers such that
  \begin{equation}\label{eq_root_change}
    I_{\gamma}(X;\bk)=(-1)^{\sum_{e\in P(\rt,\rt')}\bk(e)}I_{\gamma}(X';\bk)+\sum_{j=1}^{N}(-1)^{a_{j}}I_{\gamma}(Y_{j};\bl^{(j)})I_{\gamma}(Z_{j};\bh^{(j)}).
\end{equation}
\end{proposition}
\begin{proof}
  First, we consider the case where $\rt$ is adjacent to $\rt'$. Take a sufficienly small $z>0$. If $\bk(\{\rt,\rt'\})=0$, it is obvious that $I_{\gamma_{z}}(X;\bk)=I_{\gamma_{z}}(X';\bk)$ by definition and thus the sum on the right-hand side of \eqref{eq_root_change} is $0$. Assume $\bk(\{\rt,\rt'\})=1$. Then the difference between $\Delta(X;\bk)$ and $\Delta(X';\bk)$ is just the order of variables corresponding to $\rt$ and $\rt'$. Therefore we have
  \begin{align}
    I_{\gamma_{z}}(X;\bk)
    &=\int_{\bt\in\Delta(X;\bk)}\prod_{v\in V}F_{v,\bk_{v}^{X}}(\gamma_{z}(\bt))\\
    &=\int_{\bt\in(\Delta(Y;\bl)\times\Delta(Z;\bh))\setminus\Delta(X';\bk)}\prod_{v\in V}F_{v,\bk_{v}^{X}}(\gamma_{z}(\bt))\\
    &=I_{\gamma_{z}}(Y;\bl)I_{\gamma_{z}}(Z;\bl)-I_{\gamma_{z}}(X';\bk),
  \end{align}
  where $Y$ (resp.~$Z$) is the connected component containing $\rt$ (resp.~$\rt'$) when the edge $\{\rt,\rt'\}$ is removed from $X$ and $\bl$ (resp.~$\bh$) is the restriction of $\bk$ to the set of edges of $Y$ (resp.~$Z$). Retake $z>0$ to let all trees appearing $\gamma_{z}$-admissible if necessary. Taking the constant terms of the asymptotic expansions around $z=0$ shows the assertion. When $\bk(\{\rt,\rt'\})\ge 2$, what we should prove is a consequence of the case where $\bk=(\{\rt,\rt'\})$ by using Proposition \ref{prop_p_contraction_2}. We obtain the general assertion by repeating the adjacent case.
\end{proof}
\subsection{\texorpdfstring{$\bp$-adic finite multiple polylogarithms}{A-hat multiple polylogarithms}}
\begin{definition}[Adelic ring; {\cite[Definition 2.2]{seki19}}]
  Let $R$ be a commutative ring, $\Sigma$ a subset of $R$ and $n$ a positive integer. We put
  \[\cA_{n,R}^{\Sigma}\coloneqq\left(\prod_{a\in\Sigma}R/a^{n}R\right)\Biggm/\left(\bigoplus_{a\in\Sigma}R/a^{n}R\right).\]
  Then $\{\cA_{n,R}^{\Sigma}\}_{n=1,2,\ldots}$ equipped with natural projections becomes a projective system and thus we define
  \[\hcA_{R}^{\Sigma}\coloneqq\varprojlim_{n}\cA_{n,R}^{\Sigma}.\]
  This ring is equipped with the projective limit topology, where each $\cA_{n,R}$ is given the discrete topology. Then one has the (not necessarily continuous) natural projection
  \[\pi_{R}^{\Sigma}\colon\prod_{a\in\Sigma}\left(\varprojlim_{n}R/a^{n}R\right)\to\hcA_{R}^{\Sigma},\]
  according to \cite[Lemma 2.3]{seki19}. 
\end{definition}
\begin{remark}
  In the original definitions of $\hcA^{\Sigma}_{n,R}$ and $\hcA^{\Sigma}_{R}$ of \cite{seki19}, the superscript $\Sigma$ means an arbitrary family of ideals of $R$. However we slightly changed the definition because only the case where every element of $\Sigma$ is a principal ideal appears in this paper.
\end{remark}
In the rest of this section, fix a positive integer $N$ and an element $\alpha\in\bbZ/N\bbZ$ and put
\[\cP(N;\alpha)\coloneqq\{p\in\cP\mid p\equiv\alpha~(\mod~N)\}.\]
We call the element $\bp\coloneqq\pi_{\bbZ[\bx]}^{\cP(N;\alpha)}((p)_{p\in\cP(N;\alpha)})$ the \emph{infinitely large prime} (see \cite{seki17} for the case where $r=0$ and $N=1$). We see that $\hcA_{\bbZ[\bx]}^{\cP(N;\alpha)}$ is complete with respect to the $\bp$-adic topology via \cite[Lemma 2.5]{seki19}. 
\begin{definition}\label{def_p_fmpl}
  Let $X=(V,E,\rt,C)$ be a $\{0,1,x_{1}^{-1},\ldots,x_{r}^{-1}\}$-colored rooted tree and $\bk$ an index on $X$. For a positive integer $M$, we define a polynomial $\fmpl_{<M}(X;\bk)\in\bbQ[\bx]$ by
  \begin{equation}\label{eq_ov_mpl}
    (-1)^{|V|-|C^{-1}(\{0\})|}\frac{d}{dz}I_{\dch_{0,z}}(X;\bk)=\sum_{M=1}^{\infty}\fmpl_{<M}(X;\bk)z^{M-1}.
  \end{equation}
  Then we define the \emph{$\bp$-adic finite multiple polylogarithm} (\emph{$\bp$-adic FMPL}) \emph{associated with $X$} by
  \[\fmpl_{\hcA,\alpha}(X;\bk)\coloneqq\pi_{\bbZ[\bx]}^{\cP(N;\alpha)}((\fmpl_{<p}(X;\bk))_{p\in\cP(N;\alpha)})\in\hcA_{\bbZ[\bx]}^{\cP(N;\alpha)}.\]
\end{definition}
\begin{example}
  Assume that $X$ is linear in Definition \ref{def_p_fmpl} and take the same symbols as Example \ref{ex_linear}. When $C(\rt)=1$, we can write 
    \[\fmpl_{<M}(X;\bk)=\fmpl_{<M}\vv{C(v_{1}),\ldots,C(v_{r})}{k_{1},\ldots,k_{r}}\coloneqq\sum_{0<n_{1}<\cdots<n_{r}<M}\frac{C(v_{1})^{n_{1}}C(v_{2})^{n_{2}-n_{1}}\cdots C(v_{r})^{n_{r}-n_{r-1}}}{n_{1}^{k_{1}}\cdots n_{r}^{k_{r}}}\]
  and
    \[\fmpl_{\hcA,\alpha}(X;\bk)=\fmpl_{\hcA,\alpha}\vv{C(v_{1}),\ldots,C(v_{r})}{k_{1},\ldots,k_{r}}\coloneqq\pi_{\bbZ[\bx]}^{\cP(N;\alpha)}\left(\left(\fmpl_{<p}\vv{C(v_{1}),\ldots,C(v_{r})}{k_{1},\ldots,k_{r}}\right)_{p\in\cP}\right)\in\hcA_{\bbZ[\bx]}^{\cP(N;\alpha)}.\]
We call $\fmpl_{\hcA,\alpha}\vv{x_{1},\ldots,x_{r}}{k_{1},\ldots,x_{r}}$ the \emph{$\bp$-adic finite multiple polylogarithm}. For later convenience, we extend $\bbQ$-linearly these maps $\fmpl_{\hcA,\alpha}$ and $\fmpl_{<M}$ for a positive integer $M$. We also remark that, even the condition $C(\rt)=1$ is not satisfied, the associated polylogarithm is written as
\[\fmpl_{\hcA}(X;\bk)=C(\rt)^{\bp}\pi_{\bbZ[\bx]}^{\cP(N;\alpha)}\left(\left(\fmpl_{<p}\vv{C(v_{1})/C(\rt),\ldots,C(v_{r})/C(\rt)}{k_{1},\ldots,k_{r}}\right)_{p\in\cP}\right)\in\hcA_{\bbZ[\bx,C(\rt),C(\rt)^{-1}]}^{\cP(N;\alpha)},\]
as long as $C(\rt)\neq 0$. Here $a^{\bp}$ denotes the natural projection of $(a^{p})_{p\in\cP(N;\alpha)}$.
\end{example}
\begin{remark}
  When $N=1$, the definition of $\fmpl_{\hcA,\alpha}\vv{\bx}{\bk}$ is due to Seki \cite[Definition 2.7]{seki19} (though the order of arguments is different), who call it the \emph{finite shuffle multiple polylogarithm}.
\end{remark}
\begin{proposition}\label{prop_p_fmpl}
  Let $0<z<1$ be a real number, $X=(V,E,\rt,C)$ a $\dch_{0,z}$-admissible $\{0,1\}$-colored rooted tree and $\bk$ an essentially positive index on $X$. Assume that $N=1$. Then the element $\fmpl_{\hcA,\alpha}(X;\bk)$ coincides with $\zeta_{\cA}(X';\bk)$ modulo $\bp$. Here $X'$ stands for the $2$-colored rooted tree corresponding to $X$.
\end{proposition}
\begin{proof}
  It is sufficient to show 
  \begin{equation}\label{eq_agreement_with_Ono}
    (-1)^{|V_{\bullet}|}\frac{d}{dz}I_{\dch_{0,z}}(X;\bk)=\sum_{M=1}^{\infty}\zeta_{<M}(X';\bk)z^{M-1},
  \end{equation}
  where $V_{\bullet}\coloneqq C^{-1}(\{1\})$. First we prove \eqref{eq_agreement_with_Ono} in case $(X;\bk)$ is $\dch_{0,z}$-harvestable. For a positive integer $M$, let $Z_{<M}\colon\fH^{1}\to\bbQ$ be the $\bbQ$-linear map determined by $1\mapsto 1$ and
  \[e_{1}e_{0}^{k_{1}-1}\cdots e_{1}e_{0}^{k_{r}-1}\mapsto\zeta_{<M}(k_{1},\ldots,k_{r})\qquad (k_{1},\ldots,k_{r}\in\bbZ_{\ge 1}).\]
  The proof of \cite[Proposition 3.2]{ono17} showed
  \[\zeta_{<p}(X';\bk)=Z_{<p}(O_{(X';\bk)})\]
  for each prime number $p$, however it remains available to prove
  \[\zeta_{<M}(X';\bk)=Z_{<M}(O_{(X';\bk)})\]
  for a general positive integer $M$. Thus, taking the generating function, we have
  \[\sum_{M=1}^{\infty}(\zeta_{<M}(X';\bk))z^{M-1}=\frac{(-1)^{|V_{\bullet}|-1}}{1-z}I_{\dch_{0,z}}(\tzero;O_{(X';\bk)};(z,-z))=(-1)^{|V_{\bullet}|}\frac{d}{dz}I_{\dch_{0,z}}(\tzero;w_{(X;\bk)};(z,-z)).\]
  It follows from Theorem \ref{thm_shuffle_theorem} that the right-hand side agrees with the left-hand side of \eqref{eq_agreement_with_Ono}. Next we deal with the general case. Since the construction of $(X_{\har};\bk_{\har})$ is same as that of the harvestable form $(X'_{\h};\bk_{\h})$ of $(X';\bk)$ in \cite[Definition 2.10]{ono17} when no change of roots appears, $(X_{\har};\bk_{\har})$ agrees with the $\dch_{0,z}$-admissible $\{0,1\}$-colored pair $(\widetilde{X'_{\h}};\bk_{\h})$ corresponding to the harvestable form $(X'_{\h};\bk_{\h})$ of $(X';\bk)$. Hence we have
  \begin{align}
    (-1)^{|V_{\bullet}|}\frac{d}{dz}I_{\dch_{0,z}}(X;\bk)
    &=(-1)^{|V_{\bullet}|}\frac{d}{dz}I_{\dch_{0,z}}(X_{\har};\bk_{\har})\\
    &=(-1)^{|V_{\bullet}|}\frac{d}{dz}I_{\dch_{0,z}}(\widetilde{X'_{\h}};\bk_{\h})\\
    &=\sum_{M=1}^{\infty}\zeta_{<M}(X'_{\h};\bk_{\h})z^{M-1}\\
    &=\sum_{M=1}^{\infty}\zeta_{<M}(X';\bk)z^{M-1}.
  \end{align}
  Here we used \eqref{eq_agreement_with_Ono} in the third equality and \cite[Proposition 2.9]{ono17} in the fourth equality. Note that the original proof of \cite[Proposition 2.9]{ono17} is only written in case $M$ is a prime number, however it holds for an arbitrary positive integer $M$ since the root is not changed in our case.
\end{proof}
Denote by $\cR$ the $\bbQ$-linear space spanned by tuples of the form $\vv{a_{1},\ldots,a_{r}}{k_{1},\ldots,k_{r}}$, where $k_{1},\ldots,k_{r}$ are positive integers and $a_{1},\ldots,a_{r}\in\{1,x_{1},\ldots,x_{r}\}$. Then there exists a $\bbQ$-linear bijection determined by
\[W\colon\cR\ni\vv{a_{1},\ldots,a_{r}}{k_{1},\ldots,k_{r}}\mapsto\prod_{i=1}^{r}e_{a_{i}}e_{0}^{k_{i}-1}\in\fH_{\{0,1,x_{1},\ldots,x_{r}\}}^{1}\coloneqq\bbQ\oplus e_{1}\fH_{\{0,1,x_{1},\ldots,x_{r}\}}\oplus\bigoplus_{i=1}^{r}e_{x_{i}}\fH_{\{0,1,x_{1},\ldots,x_{r}\}}.\]
Since $\fH_{\{0,1,x_{1},\ldots,x_{r}\}}^{1}$ is a commutative $\bbQ$-algebra via the shuffle product, we can equip a commutative $\bbQ$-algebra structure on $\cR$ by introducing $\sh$ through this identification. Namely, we define $a\sh b\coloneqq W^{-1}(W(a)\sh W(b))$ for $a,b\in\cR$.
\begin{theorem}\label{thm_p_shuffle}
  Let $k_{1},\ldots,k_{r},l_{1},\ldots,l_{s}$ be positive integers. Then the equality
  \begin{align}
    &\fmpl_{\hcA,\alpha}\left(\vv{x_{1},\ldots,x_{r}}{k_{1},\ldots,k_{r}}\sh\vv{y_{1},\ldots,y_{s}}{l_{1},\ldots,l_{s}}\right)\\
    &\qquad=(-1)^{l_{1}+\cdots+l_{s}}y_{1}^{\bp}\sum_{f_{1},\ldots,f_{s}\ge 0}\left(\prod_{i=1}^{s}\binom{l_{i}+f_{i}-1}{f_{i}}\right)\fmpl_{\hcA,\alpha}\vv{x_{1}/y_{1},\ldots,x_{r}/y_{1},1/y_{1},y_{s}/y_{1},\ldots,y_{2}/y_{1}}{k_{1},\ldots,k_{r},l_{s}+f_{s},\ldots,l_{1}+f_{1}}\bp^{f_{1}+\cdots+f_{s}}
\end{align} 
holds in $\hcA^{\cP(N;\alpha)}_{\bbZ[\bx,y_{1},y_{1}^{-1},y_{2},\ldots,y_{s}]}$.
\end{theorem}
\begin{proof}
  Consider the $\{0,1,x_{1}^{-1},\ldots,x_{r}^{-1},y_{1}^{-1},\ldots,y_{s}^{-1}\}$-colored rooted tree $X$ and the index $\bk$ on $X$ consisting of the following data:
\begin{align}
  X&=(V,E,\rt,C),\\
  V&=\{v_{1},\ldots,v_{r},w_{1},\ldots,w_{s},\rt\},\\
  E&=\{e_{i}=\{v_{i},v_{i+1}\}\mid i=1,\ldots,r~(v_{r+1}\coloneqq\rt)\}\cup\{d_{i}=\{w_{i},w_{i+1}\}\mid i=1,\ldots,s~(w_{s+1}\coloneqq\rt)\},\\
  C(v)&=\begin{cases}x_{i}^{-1} & \text{if }v=v_{i},~i=1,\ldots,r,\\y_{i}^{-1} & \text{if }v=w_{i},~i=1,\ldots,s,\\ 1 & \text{if }v=\rt,\end{cases}
\end{align}
and
  \[\bk(e)=\begin{cases}k_{i} & \text{if }e=e_{i},\\ l_{i} & \text{if }e=d_{i}.\end{cases}\]
Theorem \ref{thm_shuffle_theorem} and expanding the power series show
\begin{equation}\label{eq_p_shuffle_1}
  I_{\dch_{0,z}}(X;\bk)=(-1)^{r+s+1}\int_{0}^{z}\left(\sum_{M=1}^{\infty}\fmpl_{<M}\left(\vv{x_{1},\ldots,x_{r}}{k_{1},\ldots,k_{r}}\sh\vv{y_{1},\ldots,y_{s}}{l_{1},\ldots,l_{s}}\right)z^{M-1}\right)\,dz.
\end{equation}
Write $X'=(V,E,w_{1},C)$. Running the algorithm showed in the proof of Proposition \ref{prop_root_change} explicitly and use Proposition \ref{prop_p_contraction_2}, we obtain
\[I_{\dch_{0,z}}(X;\bk)=(-1)^{l_{1}+\cdots+l_{s}}I_{\dch_{0,z}}(X';\bk)+\sum_{i=1}^{s}\sum_{j=0}^{l_{i}-1}(-1)^{l_{i+1}+\cdots+l_{s}+j}I_{\dch_{0,z}}(Y_{i,j};\bl_{i,j})I_{\dch_{0,z}}(Z_{i,j};\bh_{i,j})\]
with
\begin{align}
  Y_{i,j}&=(V_{i,j},E_{i,j},w',C_{i,j}),\\
  V_{i,j}&=\begin{cases}\{v_{1},\ldots,v_{r},w_{i+1},\ldots,w_{s},w',\rt\} & \text{if }j>0,\\ \{v_{1},\ldots,v_{r},w'=w_{i+1},w_{i+2},\ldots,w_{s},\rt\} & \text{if }j=0,\end{cases}\\
  E_{i,j}&=\begin{cases}\{e_{1},\ldots,e_{r},d_{i+1},\ldots,d_{s},d'\coloneqq\{w',w_{i+1}\}\} & \text{if }j>0,\\\{e_{1},\ldots,e_{r},d_{i+1},\ldots,d_{s}\} & \text{if }j=0,\end{cases}\\
  C_{i,j}(v)&=\begin{cases}0 & \text{if }j>0,~v=w',\\ C(v) & \text{otherwise},\end{cases}\\
  \bl_{i,j}(e)&=\begin{cases}j & \text{if }j>0,~e=d',\\ \bk(e) & \text{otherwise},\end{cases}
\end{align}
and
\begin{align}
  Z_{i,j}&=(V'_{i,j},E'_{i,j},w'',C'_{i,j}),\\
  V'_{i,j}&=\begin{cases}\{w_{1},\ldots,w_{i},w''\} & \text{if }j<l_{i}-1,\\ \{w_{1},\ldots,w_{i-1},w''=w_{i}\} & \text{if }j=l_{i}-1,\end{cases}\\
  E'_{i,j}&=\begin{cases}\{d_{1},\ldots,d_{i-1},d''\coloneqq\{w_{i},w''\}\} & \text{if }j<l_{i}-1,\\\{d_{1},\ldots,d_{i-1}\} & \text{if }j=l_{i}-1,\end{cases}\\
  C'_{i,j}(v)&=\begin{cases}0 & \text{if }j<l_{i}-1,~v=w'',\\ C(v) & \text{otherwise},\end{cases}\\
  \bh_{i,j}(e)&=\begin{cases}l_{i}-1-j & \text{if }j<l_{i}-1,~e=d'',\\ \bk(e) & \text{otherwise}.\end{cases}
\end{align}
From these data, computing the associated integrals similarly to Example \ref{ex_linear}, we have
\begin{align}
  I_{\dch_{0,z}}(Y_{i,j};\bl_{i,j})&=(-1)^{r+s+1-i}\sum_{0<m_{1}<\cdots<m_{r}<n_{s}<\cdots<n_{i}}\frac{x_{1}^{m_{1}}x_{2}^{m_{2}-m_{1}}\cdots x_{r}^{m_{r}-m_{r-1}}x_{r+1}^{n_{s}-m_{r}}y_{s}^{n_{s-1}-n_{s}}\cdots y_{i+1}^{n_{i}-n_{i+1}}}{m_{1}^{k_{1}}\cdots m_{r}^{k_{r}}n_{s}^{l_{s}}\cdots n_{i+1}^{l_{i+1}}n_{i}^{j+1}}z^{n_{i}},\\
  I_{\dch_{0,z}}(Z_{i,j};\bh_{i,j})&=(-1)^{i}\sum_{0<n_{1}<\cdots<n_{i}}\frac{y_{1}^{n_{1}}y_{2}^{n_{2}-n_{1}}\cdots y_{i}^{n_{i}-n_{i-1}}}{n_{1}^{l_{1}}\cdots n_{i-1}^{l_{i-1}}n_{i}^{l_{i}-j}}z^{n_{i}}
\end{align}
and
  \[I_{\dch_{0,z}}(X';\bk)=(-1)^{r+s+1}\int_{0}^{z}\left(\sum_{M=1}^{\infty}y_{1}^{M}\fmpl_{<M}\left(\vv{x_{1}/y_{1},\ldots,x_{r}/y_{1},1/y_{1},y_{s}/y_{1},\ldots,y_{2}/y_{1}}{k_{1},\ldots,k_{r},l_{s},\ldots,l_{1}}\right)z^{M-1}\right)\,dz.\]
Then combining them and \eqref{eq_p_shuffle_1} shows
\begin{align}
  &\fmpl_{<p}\left(\vv{x_{1},\ldots,x_{r}}{k_{1},\ldots,k_{r}}\sh\vv{y_{1},\ldots,y_{s}}{l_{1},\ldots,l_{s}}\right)\\
  &=(-1)^{l_{1}+\cdots+l_{s}}y_{1}^{p}\fmpl_{<p}\vv{x_{1}/y_{1},\ldots,x_{r}/y_{1},1/y_{1},y_{s}/y_{1},\ldots,y_{2}/y_{1}}{k_{1},\ldots,k_{r},l_{s},\ldots,l_{1}}\\
  &\qquad+p\sum_{i=1}^{s}\sum_{j=0}^{l_{i}-1}(-1)^{l_{i+1}+\cdots+l_{s}+j}\sum_{N=1}^{p-1}\fmpl_{=N}\vv{x_{1},\ldots,x_{r},1,y_{s},\ldots,y_{i+1}}{k_{1},\ldots,k_{r},l_{s},\ldots,l_{i+1},j+1}\fmpl_{=p-N}\vv{y_{1},\ldots,y_{i}}{l_{1},\ldots,l_{i-1},l_{i}-j}
\end{align}
for a prime number $p$, where we put $\fmpl_{=M}\coloneqq \fmpl_{<M+1}-\fmpl_{<M}$. Then using $y_{1}=1$ and the $p$-adic expansion
\begin{align}
  &\fmpl_{=p-N}\vv{y_{1},\ldots,y_{i}}{l_{1},\ldots,l_{i-1},l_{i}-j}\\
  &=\sum_{0<n_{1}<\cdots<n_{i-1}<p-N}\frac{y_{1}^{n_{1}}y_{2}^{n_{2}-n_{1}}\cdots y_{i}^{p-N-n_{i-1}}}{n_{1}^{l_{1}}\cdots n_{i-1}^{l_{i-1}}(p-N)^{l_{i}-j}}\\
  &=\sum_{f_{1},\ldots,f_{i}\ge 0}\sum_{N<n_{1}<\cdots<n_{i-1}<p}(-1)^{l_{1}+\cdots+l_{i-1}+l_{i}-j}\\
  &\qquad\cdot\binom{l_{1}+f_{1}-1}{f_{1}}\cdots\binom{l_{i-1}+f_{i-1}-1}{f_{i-1}}\binom{l_{i}-j+f_{i}-1}{f_{i}}\frac{y_{i}^{n_{1}-N}y_{i-1}^{n_{2}-n_{1}}\cdots y_{2}^{n_{i-1}-n_{i-2}}y_{1}^{p-n_{i-1}}}{N^{l_{i}-j+f_{i}}n_{1}^{l_{i-1}+f_{i-1}}\cdots n_{i-1}^{l_{1}+f_{1}}}p^{f_{1}+\cdots+f_{i}},
\end{align}
we obtain
\begin{align}
  &\fmpl_{<p}\left(\vv{x_{1},\ldots,x_{r}}{k_{1},\ldots,k_{r}}\sh\vv{y_{1},\ldots,y_{s}}{l_{1},\ldots,l_{s}}\right)\\
  &=(-1)^{l_{1}+\cdots+l_{s}}y_{1}^{p}\fmpl_{<p}\vv{x_{1}/y_{1},\ldots,x_{r}/y_{1},1/y_{1},y_{s}/y_{1},\ldots,y_{2}/y_{1}}{k_{1},\ldots,k_{r},l_{s},\ldots,l_{1}}\\
  &\qquad+(-1)^{l_{1}+\cdots+l_{s}}\sum_{i=1}^{s}\sum_{j=0}^{l_{i}-1}\sum_{f_{1},\ldots,f_{i}\ge 0}\binom{l_{1}+f_{1}-1}{f_{1}}\cdots\binom{l_{i-1}+f_{i-1}-1}{f_{i-1}}\binom{l_{i}-j+f_{i}-1}{f_{i}}\\
  &\qquad\cdot y_{1}^{p}\fmpl_{<p}\vv{x_{1}/y_{1},\ldots,x_{r}/y_{1},1/y_{1},y_{s}/y_{1},\ldots,y_{2}/y_{1}}{k_{1},\ldots,k_{r},l_{s},\ldots,l_{i+1},l_{i}+f_{i}+1,l_{i-1}+f_{i-1},\ldots,l_{1}+f_{1}}p^{f_{1}+\cdots+f_{i}+1}\\
  &=(-1)^{l_{1}+\cdots+l_{s}}y_{1}^{p}\fmpl_{<p}\vv{x_{1}/y_{1},\ldots,x_{r}/y_{1},1/y_{1},y_{s},\ldots,y_{2}}{k_{1},\ldots,k_{r},l_{s},\ldots,l_{1}}\\
  &\qquad+(-1)^{l_{1}+\cdots+l_{s}}\sum_{i=1}^{s}\sum_{f_{1},\ldots,f_{i}\ge 0}\binom{l_{1}+f_{1}-1}{f_{1}}\cdots\binom{l_{i-1}+f_{i-1}-1}{f_{i-1}}\binom{l_{i}+f_{i}}{f_{i}+1}\\
  &\qquad\cdot\fmpl_{<p}\vv{x_{1}/y_{1},\ldots,x_{r}/y_{1},1/y_{1},y_{s}/y_{1},\ldots,y_{2}/y_{1}}{k_{1},\ldots,k_{r},l_{s},\ldots,l_{i+1},l_{i}+f_{i}+1,l_{i-1}+f_{i-1},\ldots,l_{1}+f_{1}}p^{f_{1}+\cdots+f_{i}+1}\\
  &=(-1)^{l_{1}+\cdots+l_{s}}y_{1}^{p}\fmpl_{<p}\vv{x_{1}/y_{1},\ldots,x_{r}/y_{1},1/y_{1},y_{s}/y_{1},\ldots,y_{2}/y_{1}}{k_{1},\ldots,k_{r},l_{s},\ldots,l_{1}}\\
  &\qquad+(-1)^{l_{1}+\cdots+l_{s}}\sum_{\substack{f_{1},\ldots,f_{s}\ge 0\\ f_{1}+\cdots+f_{s}\ge 1}}\left(\prod_{i=1}^{s}\binom{l_{i}+f_{i}-1}{f_{i}}\right)y_{1}^{p}\fmpl_{<p}\binom{x_{1}/y_{1},\ldots,x_{r}/y_{1},1/y_{1},y_{s}/y_{1},\ldots,y_{2}/y_{1}}{k_{1},\ldots,k_{r},l_{s}+f_{s},\ldots,l_{1}+f_{1}}p^{f_{1}+\cdots+f_{s}}\\
  &=(-1)^{l_{1}+\cdots+l_{s}}y_{1}^{p}\sum_{f_{1},\ldots,f_{s}\ge 0}\left(\prod_{i=1}^{s}\binom{l_{i}+f_{i}-1}{f_{i}}\right)\fmpl_{<p}\binom{x_{1}/y_{1},\ldots,x_{r}/y_{1},1/y_{1},y_{s}/y_{1},\ldots,y_{2}/y_{1}}{k_{1},\ldots,k_{r},l_{s}+f_{s},\ldots,l_{1}+f_{1}}p^{f_{1}+\cdots+f_{s}}.
\end{align}
\end{proof}
Denote by $\hcA(N;\alpha)$ (resp.~$\cA_{n}(N;\alpha)$) the ring $\hcA_{R}^{\Sigma}$ (resp.~$\cA_{n,R}^{\Sigma}$) in case $R=\bbZ[X]/(X^{N}-1)$ and $\Sigma=\cP(N;\alpha)$. We write the natural projection as $\pi_{N}\coloneqq\pi^{\cP(N;\alpha)}_{\bbZ[X]/(X^{N}-1)}$ and the element $\pi_{N}((p)_{p\in\cP(N;\alpha)})$ as $\bp$ by abuse of notation. The ring $\hcA(N;\alpha)$ is $\bp$-adically complete.
\begin{remark}\label{rem_induce}
  We also easily see that, when once we determine a map $f\colon\{x_{1},\ldots,x_{r}\}\to\mu_{N}$, the continuous ring homomorphism $f\colon\hcA_{\bbZ[\bx]}^{\cP(N;\alpha)}\to\hcA(N;\alpha)$ is induced by the definition of the projective limit.
\end{remark}
\begin{definition}\label{def_p_fmlv}
  For $\vv{\etab}{\bk}=\vv{\eta_{1},\ldots,\eta_{r}}{k_{1},\ldots,k_{r}}\in\mu_{N}^{r}\times\bbZ_{\ge 1}^{r}$, we put
  \[L_{<M}\vv{\etab}{\bk}\coloneqq\sum_{0<n_{1}<\cdots<n_{r}<M}\frac{\eta_{1}^{n_{1}}\eta_{2}^{n_{2}-n_{1}}\cdots\eta_{r}^{n_{r}-n_{r-1}}}{n_{1}^{k_{1}}\cdots n_{r}^{k_{r}}}\]
  and define the \emph{$\bp$-adic finite multiple $L$-value} as
  \[L_{\hcA,\alpha}\vv{\etab}{\bk}\coloneqq\pi_{N}\left(\left(L_{<p}\vv{\etab}{\bk}\right)_{p\in\cP(N;\alpha)}\right)\in\hcA(N;\alpha).\]
\end{definition}
\begin{remark}
  The image of $L_{\hcA,\alpha}\vv{\etab}{\bk}$ under the natural projection $\hcA(N;\alpha)\to\cA_{1}(N;\alpha)\simeq\hcA(N;\alpha)/\bp\hcA(N;\alpha)$ is essentially Tasaka's finite colored multiple zeta value $L^{\cA}_{\alpha}\vv{\etab}{\bk}$ (\cite[Definition 4.1]{tasaka21}) with an appropriate adjusting of arguments.
\end{remark}
For a positive integer $N$, we define
\[\cR_{N}\coloneqq\mathrm{span}_{\bbQ}\left(\bigcup_{r=0}^{\infty}(\mu_{N}^{r}\times\bbZ_{\ge 1}^{r})\right).\]
\begin{corollary}\label{cor_p_shuffle}
  Let $k_{1},\ldots,k_{r},l_{1},\ldots,l_{s}$ be positive integers and $\eta_{1},\ldots,\eta_{r},\xi_{1},\ldots,\xi_{s}$ elements of $\mu_{N}$. Then we have  
  \begin{align}
    &L_{\hcA,\alpha}\left(\vv{\eta_{1},\ldots,\eta_{r}}{k_{1},\ldots,k_{r}}\sh\vv{\xi_{1},\ldots,\xi_{s}}{l_{1},\ldots,l_{s}}\right)\\
    &\qquad=(-1)^{l_{1}+\cdots+l_{s}}\xi_{1}^{\bp}\sum_{f_{1},\ldots,f_{s}\ge 0}\left(\prod_{i=1}^{s}\binom{l_{i}+f_{i}-1}{f_{i}}\right)L_{\hcA,\alpha}\vv{\eta_{1}/\xi_{1},\ldots,\eta_{r}/\xi_{1},1/\xi_{1},\xi_{s}/\xi_{1},\ldots,\xi_{2}/\xi_{1}}{k_{1},\ldots,k_{r},l_{s}+f_{s},\ldots,l_{1}+f_{1}}\bp^{f_{1}+\cdots+f_{s}}.
  \end{align}
  Here we extend $L_{\hcA,\alpha}$ to $\cR_{N}$.
\end{corollary}
\begin{proof}
  Determine $f\colon\{x_{1},\ldots,x_{r},y_{1},\ldots,y_{s}\}\to\mu_{N}$ by $f(x_{i})=\eta_{i}$ for $i=1,\ldots,r$ and $f(y_{j})=\xi_{j}$ for $j=1,\ldots,s$. Then we get this corollary by applying the induced map $f$ in the sense of Remark \ref{rem_induce} to Theorem \ref{thm_p_shuffle}.
\end{proof}
The following corollary is immediate by putting $N=1$ and $r=0$ in Theorem \ref{thm_p_shuffle}.
\begin{corollary}[{\cite[Theorem 3.1]{seki19}}]
  Let $k_{1},\ldots,k_{r}$ be positive integers. Then the equality
  \begin{align}
    \fmpl_{\hcA,0}\vv{x_{1},\ldots,x_{r}}{k_{1},\ldots,k_{r}}=(-1)^{k_{1}+\cdots+k_{r}}x_{1}^{\bp}\sum_{f_{1},\ldots,f_{r}\ge 0}\left(\prod_{i=1}^{r}\binom{k_{i}+f_{i}-1}{f_{i}}\right)\fmpl_{\hcA,0}\vv{1/x_{1},x_{r}/x_{1},\ldots,x_{2}/x_{1}}{k_{r}+f_{r},\ldots,k_{1}+f_{1}}\bp^{f_{1}+\cdots+f_{r}}
\end{align} 
holds in $\hcA^{\cP(1;0)}_{\bbZ[\bx,\bx^{-1}]}$.
\end{corollary}
\subsection{\texorpdfstring{$t$-adic symmetric multiple polylogarithms}{S-hat symmetric multiple polylogarithms}}
\begin{definition}\label{def_t_smpl}
    For $\bk=(k_{1},\ldots,k_{r})\in\bbZ_{\ge 1}^{r}$, we define the \emph{$t$-adic symmetric multiple polylogarithm} (\emph{$t$-adic SMPL}) as
      \begin{align}
      \fmpl_{\hcS,\alpha}\vv{\bx}{\bk}&\coloneqq\sum_{i=0}^{r}(-1)^{k_{i+1}+\cdots+k_{r}}x_{i+1}^{\alpha}\Li^{\sh}\vv{x_{1}/x_{i+1},\ldots,x_{i}/x_{i+1}}{k_{1},\ldots,k_{i}}\\
        &\qquad\cdot\sum_{f_{i+1},\ldots,f_{r}\ge 0}\left(\prod_{j=i+1}^{r}\binom{k_{j}+f_{j}-1}{f_{j}}\right)\Li^{\sh}\vv{1/x_{i+1},x_{r}/x_{i+1},\ldots,x_{i+2}/x_{i+1}}{k_{r}+f_{r},\ldots,k_{i+1}+f_{i+1}}t^{f_{i+1}+\cdots+f_{r}}\in\bbC\jump{\bx,\bx^{-1}}\jump{t},
      \end{align}
      where we write
      \begin{equation}\label{eq_polylog_ii}
        \Li^{\sh}\vv{x_{1},\ldots,x_{r}}{k_{1},\ldots,k_{r}}\coloneqq\sum_{0<n_{1}<\cdots<n_{r}}\frac{x_{1}^{n_{1}}x_{2}^{n_{2}-n_{1}}\cdots x_{r}^{n_{r}-n_{r-1}}}{n_{1}^{k_{1}}\cdots n_{r}^{k_{r}}}=(-1)^{r}I_{\dch}(\tzero;e_{1/x_{1}}e_{0}^{k_{1}-1}\cdots e_{1/x_{r}}e_{0}^{k_{r}-1};\tone).
      \end{equation}
    This definition gives the $\bbQ$-linear map $\fmpl_{\hcS,\alpha}\colon\cR\to\bbC\jump{\bx,\bx^{-1}}\jump{t}$.
\end{definition}
\begin{definition}\label{def_t_rsmpl}
  Let $X=(V,E,\rt,C)$ be a $\{0,1,x_{1}^{-1},\ldots,x_{r}^{-1}\}$-colored rooted tree and $\bk$ an essentially positive index on $X$. For $u\in\{x_{1},\ldots,x_{r}\}$ and a non-negative integer $n$, using a new vertex $w$, we define a $\{0,x_{1},\ldots,x_{r}\}\cup\{x_{i}x_{j}^{-1}\mid i,j\in\{1,\ldots,r\}\}$-colored pair $(X'_{u};\bk_{n})$ by
    \begin{align}
      C'_{u}(v)&=\begin{cases}0 & \text{if }v=w,\\ C(v)u & \text{otherwise},\end{cases}\\
      X'_{u}&=(V\cup\{w\},E\cup\{\{\rt,w\}\},w,C'_{u}),\\
      \bk_{n}(e)&=\begin{cases}n & \text{if }e=\{\rt,w\},\\ \bk(e) & \text{otherwise}.\end{cases}
    \end{align}
    Then we define the \emph{$t$-adic refined symmetric multiple polylogarithm associated with $X$} as
    \[\fmpl_{\hcRS,\alpha}(X;\bk)\coloneqq\frac{(-1)^{|V|-|C^{-1}(\{0\})|-1}}{2\pi i}\sum_{n=0}^{\infty}\sum_{u\in\{x_{1},\ldots,x_{r}\}}u^{\alpha}I_{\beta}(X'_{u};\bk_{n})t^{n}\in\bbC\jump{\bx,\bx^{-1}}\jump{t}.\]
\end{definition}
\begin{proposition}\label{prop_t_smpl}
  Let $X=(V,E,\rt,C)$ be a $\beta$-admissible $\{0,1,x_{1}^{-1},\ldots,x_{r}^{-1}\}$-colored rooted tree and $\bk$ an essentially positive index on $X$. Then we have the followings.
  \begin{enumerate}
    \item\label{t_smpl_1} The element $\fmpl_{\hcRS,\alpha}(X;\bk)$ is in $\bbQ\jump{\bx,\bx^{-1}}[2\pi i]\jump{t}$.
    \item\label{t_smpl_2} When $r=0$, the value $\fmpl_{\hcRS,\alpha}(X;\bk)$ coincides with $\zeta_{\hcS}(X'';\bk)$ modulo $2\pi i$. Here $X''$ stands for the $2$-colored rooted tree corresponding to $X$.
    \item\label{t_smpl_3} When $X$ is linear as in Example \ref{ex_linear}, the congruence
    \[\fmpl_{\hcRS,\alpha}(X;\bk)\equiv C(\rt)^{-\alpha}\fmpl_{\hcS,\alpha}\vv{C(\rt)/C(v_{1}),\ldots,C(\rt)/C(v_{r})}{k_{1},\ldots,k_{r}}\qquad (\mod~2\pi i)\]
    holds.
    \item\label{t_smpl_4} When $X$ is linear as in Example \ref{ex_linear} and $r=0$, the value $\fmpl_{\hcRS,\alpha}(X;\bk)$ coincides with $\zeta_{\cRS}(k_{1},\ldots,k_{r})$ modulo $t$.
  \end{enumerate}
\end{proposition}
\begin{proof}
  The claim \eqref{t_smpl_4} is proved by a similar computation to Example \ref{ex_linear}. We prove \eqref{t_smpl_1} and \eqref{t_smpl_3}. Take the same symbols about $(X;\bk)$ as Example \ref{ex_linear} and determines elements $a_{1},\ldots,a_{l}\in\{0,x_{1},\ldots,x_{r+1}\}\cup\{x_{i}x_{j}^{-1}\mid i,j\in\{1,\ldots,r+1\}\}$ by
  \[e_{a_{1}}\cdots e_{a_{l}}\coloneqq e_{u/x_{1}}e_{0}^{k_{1}-1}\cdots e_{u/x_{r}}e_{0}^{k_{r}-1}e_{u/x_{r+1}}e_{0}^{n},\]
  for $u\in\{x_{1},\ldots,x_{r+1}\}$. By the path composition formula and the reversal formula (Proposition \ref{prop_ii_properties} \eqref{pathcon} and \eqref{reversal}), we have
  \begin{align}
    &I_{\beta}(\tzero;e_{u/x_{1}}e_{0}^{k_{1}-1}\cdots e_{u/x_{r}}e_{0}^{k_{r}-1}e_{u/x_{r+1}}e_{0}^{n};\tzero)\\
    &=\sum_{0\le j\le s\le l}I_{\dch}(\tzero;e_{a_{1}}\cdots e_{a_{j}};\tone)I_{c}(\tone;e_{a_{j+1}}\cdots e_{a_{s}};\tone)I_{\dch^{-1}}(\tone;e_{a_{s+1}}\cdots e_{a_{l}};\tzero)\\
    &=\sum_{0\le j\le s\le l}(-1)^{l-s}I_{\dch}(\tzero;e_{a_{1}}\cdots e_{a_{j}};\tone)I_{c}(\tone;e_{a_{j+1}}\cdots e_{a_{s}};\tone)I_{\dch}(\tzero;e_{a_{l}}\cdots e_{a_{s+1}};\tone).
  \end{align} 
  Since
  \[I_{c}(\tone;e_{a_{j+1}}\cdots e_{a_{s}};\tone)=\begin{cases}\dfrac{(2\pi i)^{s-j}}{(s-j)!} & \text{if }a_{j+1}=\cdots=a_{s}=1,\\ 0 & \text{otherwise},\end{cases}\]
  we obtain
  \begin{align}
    &\sum_{u\in\{x_{1},\ldots,x_{r+1}\}}u^{\alpha}I_{\beta}(\tzero;e_{u/x_{1}}e_{0}^{k_{1}-1}\cdots e_{u/x_{r}}e_{0}^{k_{r}-1}e_{u/x_{r+1}}e_{0}^{n};\tzero)\\
    &=\sum_{u\in\{x_{1},\ldots,x_{r+1}\}}u^{\alpha}\sum_{\substack{0\le j<s\le l\\ a_{j+1}=\cdots=a_{s}=1}}(-1)^{l-s}\frac{(2\pi i)^{s-j}}{(s-j)!}I_{\dch}(\tzero;e_{a_{1}}\cdots e_{a_{j}};\tone)I_{\dch}(\tzero;e_{a_{l}}\cdots e_{a_{s+1}};\tone)\\
    &=\sum_{u\in\{x_{1},\ldots,x_{r+1}\}}u^{\alpha}\sum_{\substack{0\le j\le s\le r\\ k_{j+1}=\cdots=k_{s}=1\\ x_{j+1}=\cdots=x_{s+1}=u}}(-1)^{k_{j+1}+\cdots+k_{r}+n+r-s+j}\frac{(2\pi i)^{s-j+1}}{(s-j+1)!}I_{\dch}(\tzero;e_{u/x_{1}}e_{0}^{k_{1}-1}\cdots e_{u/x_{j}}e_{0}^{k_{j}-1};\tone)\\
    &\qquad\cdot I_{\dch}(\tzero;e_{0}^{n}e_{u/x_{r+1}}e_{0}^{k_{r}-1}e_{u/x_{r}}e_{0}^{k_{r-1}-1}\cdots e_{u/x_{s+2}}e_{0}^{k_{s+1}-1};\tone)\\
    &=(-1)^{r}2\pi i\sum_{\substack{0\le j\le s\le r\\ k_{j+1}=\cdots=k_{s}=1\\ x_{j+1}=\cdots=x_{s+1}}}(-1)^{k_{s+1}+\cdots+k_{r}}x_{j+1}^{\alpha}\frac{(-2\pi i)^{s-j}}{(s-j+1)!}\Li^{\sh}\vv{x_{1}/x_{j+1},\ldots,x_{j}/x_{j+1}}{k_{1},\ldots,k_{j}}\\
    &\qquad\cdot\sum_{\substack{f_{s+1},\ldots,f_{r}\ge 0\\ f_{s+1}+\cdots+f_{r}=n}}\left(\prod_{a=s+1}^{r}\binom{k_{a}+f_{a}-1}{f_{a}}\right)\Li^{\sh}\vv{x_{r+1}/x_{s+1},x_{r}/x_{s+1},\ldots,x_{s+2}/x_{s+1}}{k_{r}+f_{r},\ldots,k_{s+1}+f_{s+1}}.
  \end{align} 
  Note that the term where $j=s$ in the first equality and the terms where $s>r$ in the second equality vanish by the path composition formula. This expression shows \eqref{t_smpl_1}. Moroever, \eqref{t_smpl_3} also follows as
  \begin{align}
    \fmpl_{\hcRS,\alpha}(X;\bk)
    &=\frac{(-1)^{r}}{2\pi i}\sum_{n=0}^{\infty}\sum_{u\in\{x_{1},\ldots,x_{r}\}}u^{\alpha}I_{\beta}(\tzero;e_{u/x_{1}}e_{0}^{k_{1}-1}\cdots e_{u/x_{r}}e_{0}^{k_{r}-1}e_{u/x_{r+1}}e_{0}^{n};\tzero)(-t)^{n}\\
    &=\sum_{\substack{0\le j\le s\le r\\ k_{j+1}=\cdots=k_{s}=1\\ x_{j+1}=\cdots=x_{s+1}}}(-1)^{k_{s+1}+\cdots+k_{r}}x_{j+1}^{\alpha}\frac{(-2\pi i)^{s-j}}{(s-j+1)!}\Li^{\sh}\vv{x_{1}/x_{j+1},\ldots,x_{j}/x_{j+1}}{k_{1},\ldots,k_{j}}\\
    &\qquad\cdot\sum_{f_{k+1},\ldots,f_{s}\ge 0}\left(\prod_{a=k+1}^{s}\binom{l_{a}+f_{a}-1}{f_{a}}\right)\Li^{\sh}\vv{x_{r+1}/x_{s+1},x_{r}/x_{s+1},\ldots,x_{s+2}/x_{s+1}}{l_{s}+f_{s},\ldots,l_{k+1}+f_{k+1}}t^{f_{k+1}+\cdots+f_{s}}\\
    &\equiv\sum_{j=0}^{r}(-1)^{k_{j+1}+\cdots+k_{r}}x_{j+1}^{\alpha}\Li^{\sh}\vv{x_{1}/x_{j+1},\ldots,x_{j}/x_{j+1}}{k_{1},\ldots,k_{j}}\\
    &\qquad\cdot\sum_{f_{j+1},\ldots,f_{r}\ge 0}\left(\prod_{a=j+1}^{r}\binom{k_{a}+f_{a}-1}{f_{a}}\right)\Li^{\sh}\vv{x_{r+1}/x_{j+1},\ldots,x_{j+2}/x_{j+1}}{k_{r}+f_{r},\ldots,k_{j+1}+f_{j+1}}t^{f_{j+1}+\cdots+f_{r}}\qquad (\mod~2\pi i)\\
    &=x_{r+1}^{\alpha}\fmpl_{\hcS,\alpha}\vv{x_{1}/x_{r+1},\ldots,x_{r}/x_{r+1}}{k_{1},\ldots,k_{r}}.
\end{align}
Next we prove \eqref{t_smpl_2}. Take the same symbols as Definition \ref{def_t_rsmpl} and $0<z'<1/2$ such that $\gamma(t)\notin D$ and $\gamma(u)\neq C(\rt)$ for $t\in (0,z')$ and $u\in (1-z',1)$. Then $X'$ is $\beta_{z}$-admissible for $0<z<z'$. Following the algorithm in the proof of Proposition \ref{prop_p_representation_algorithm}, we see that
  \[w_{(X'_{1,\har};(\bk_{n})_{\har})}=w_{(X_{\har};\bk_{\har})}e_{0}^{n}\]
  for $n\ge 0$. Similarly as the proof of Proposition \ref{prop_p_fmpl}, it is also equal to $O_{(X''_{\h};\bk_{\h})}e_{1}e_{0}^{n}$. Since $O_{(X''_{\h};\bk_{\h})}$ is an element of $\fH^{1}$ and \cite[Theorem 3.12]{osy21} ($\rt\in V_{\bullet}$ by the $\beta$-admissibility), it suffices to show
  \[\frac{(-1)^{s}}{2\pi i}\sum_{n=0}^{\infty}I_{\beta}(\tzero;e_{1}e_{0}^{l_{1}-1}\cdots e_{1}e_{0}^{l_{s}-1}e_{1}e_{0}^{n};\tzero)t^{n}\equiv\zeta_{\hcS}(l_{1},\ldots,l_{s})\qquad (\mod~2\pi i)\]
  for $l_{1},\ldots,l_{s}\in\bbZ_{\ge 1}$. This claim follows from \eqref{t_smpl_3}.
\end{proof}
\begin{theorem}\label{thm_t_shuffle}
  Let $k_{1},\ldots,k_{r},l_{1},\ldots,l_{s}$ be positive integers. Then the equality
  \begin{align}
    &\fmpl_{\hcS,\alpha}\left(\vv{x_{1},\ldots,x_{r}}{k_{1},\ldots,k_{r}}\sh\vv{y_{1},\ldots,y_{s}}{l_{1},\ldots,l_{s}}\right)\\
    &\qquad=(-1)^{l_{1}+\cdots+l_{s}}y_{1}^{\alpha}\sum_{f_{1},\ldots,f_{s}\ge 0}\left(\prod_{i=1}^{s}\binom{l_{i}+f_{i}-1}{f_{i}}\right)\fmpl_{\hcS,\alpha}\vv{x_{1}/y_{1},\ldots,x_{r}/y_{1},1/y_{1},y_{s}/y_{1},\ldots,y_{2}/y_{1}}{k_{1},\ldots,k_{r},l_{s}+f_{s},\ldots,l_{1}+f_{1}}t^{f_{1}+\cdots+f_{s}}
\end{align} 
holds in $\bbC\jump{\bx,\bx^{-1},y_{1},\ldots,y_{s},y_{1}^{-1},\ldots,y_{s}^{-1}}\jump{t}$.
\end{theorem}
\begin{proof}
First we put
  \[D\coloneqq\{1,x_{1},\ldots,x_{r},y_{1},\ldots,y_{s}\}\]
  and
  \[D'\coloneqq\{0\}\cup \{a/b\mid a,b\in D\}.\]
Consider the $\{0\}\cup\{a\mid a^{-1}\in D\}$-colored rooted tree $X$ and the index $\bk$ on $X$ consisting of the following data:
  \begin{align}
    X&=(V,E,\rt,C),\\
    V&=\{v_{1},\ldots,v_{r},w_{1},\ldots,w_{s},\rt\},\\
    E&=\{e_{i}=\{v_{i},v_{i+1}\}\mid i=1,\ldots,r~(v_{r+1}\coloneqq\rt)\}\cup\{d_{i}=\{w_{i},w_{i+1}\}\mid i=1,\ldots,s~(w_{s+1}\coloneqq\rt)\},\\
    C(v)&=\begin{cases}x_{i}^{-1} & \text{if }v=v_{i},~i=1,\ldots,r,\\y_{i}^{-1} & (v=w_{i},~i=1,\ldots,s),\\ 1 & \text{if }v=\rt,\end{cases}
  \end{align}
  and
    \[\bk(e)=\begin{cases}k_{i} & \text{if }e=e_{i},\\ l_{i} & \text{if }e=d_{i}.\end{cases}\]
  Take the same symbols as Definition \ref{def_t_rsmpl}. Proposition \ref{prop_p_representation_algorithm} and Theorem \ref{thm_shuffle_theorem} show
  \begin{equation}\label{eq_t_shuffle_1}
    I_{\beta}(X'_{u};\bk_{n})=I_{\beta}(\tzero;(e_{u/x_{1}}e_{0}^{k_{1}-1}\cdots e_{u/x_{r}}e_{0}^{k_{r}-1}\sh e_{u/y_{1}}e_{0}^{l_{1}-1}\cdots e_{u/y_{s}}e_{0}^{l_{s}-1})e_{u}e_{0}^{n};\tzero),
  \end{equation}
  for $u\in D$. Write $X''_{u}\coloneqq (V\cup\{w\},E\cup\{\{\rt,w\}\},w_{1},C'_{u})$.
  By using Propositions \ref{prop_p_contraction_2} and \ref{prop_root_change}, there exist two families $\{(Y_{j};\bl^{(j)})\}_{j=1,\ldots,M}$ and $\{(Z_{j};\bh^{(j)})\}_{j=1,\ldots,M}$ of $D'$-colored pairs and a sequence $\{a_{j}\}_{j=1,\ldots,M}$ of integers such that
  \[I_{\beta}(X'_{u};\bk_{n})=(-1)^{l_{1}+\cdots+l_{s}}I_{\beta}(X''_{\eta};\bk_{n})+\sum_{j=1}^{M}(-1)^{a_{j}}I_{\beta}(Y_{j};\bl^{(j)})I_{\beta}(Z_{j};\bh^{(j)}).\]
  Since $I_{\beta}(Y;\bl)/2\pi i$ is an element of $\bbQ\jump{\bx,\bx^{-1}}[2\pi i]$ for every $\{0,1,x_{1}^{-1},\ldots,x_{r}^{-1}\}$-colored rooted tree $Y$ and an essentially positive index $\bl$ on $Y$ (by using Theorem \ref{thm_shuffle_theorem} and the path composition formula as the proof of Proposition \ref{prop_t_smpl} \eqref{t_smpl_1}), we have
  \begin{align}\label{eq_t_shuffle_2}
    \begin{split}
    \fmpl_{\hcRS,\alpha}(X;\bk)
    &=\frac{(-1)^{r+s}}{2\pi i}\sum_{n=0}^{\infty}\sum_{u\in D}u^{\alpha}I_{\beta}(X'_{u};\bk_{n})t^{n}\\
    &\equiv\frac{(-1)^{r+s+l_{1}+\cdots+l_{s}}}{2\pi i}\sum_{n=0}^{\infty}\sum_{u\in D}I_{\beta}(X''_{u};\bk_{n})(-t)^{n}\qquad (\mod~2\pi i).
    \end{split}
  \end{align}
  Take $0<z<1/2$ such that $(X'';\bk_{n})$ is $\beta_{z}$-admissible. By the algorithm of Proposition \ref{prop_p_representation_algorithm} and Theorem \ref{thm_shuffle_theorem}, for $u\in D$ we compute
  \begin{align}
    &I_{\beta_{z}}(X''_{u};\bk_{n})\\
    &=I_{\beta_{z}}\left((\beta(z),(1-2z)\beta'(z));(e_{0}^{n}\sh e_{u/x_{1}}e_{0}^{k_{1}-1}\cdots e_{u/x_{r}}e_{0}^{k_{r}-1})\right.\\
    &\qquad\cdot \left.e_{u}e_{0}^{l_{s}-1}e_{u/y_{s}}\cdots e_{0}^{l_{1}-1}e_{u/x_{1}};(\beta(1-z),(2z-1)\beta'(1-z))\right)\\
    &=\sum_{j=0}^{n}\sum_{\substack{f_{1},\ldots,f_{r}\ge 0\\ f_{1}+\cdots+f_{r}=n-j}}\left(\prod_{a=1}^{r}\binom{k_{a}+f_{a}-1}{f_{a}}\right)I_{\beta_{z}}\left((\beta(z),(1-2z)\beta'(z));e_{0}^{j}e_{u/x_{1}}e_{0}^{k_{1}+f_{1}-1}\cdots e_{u/x_{r}}e_{0}^{k_{r}+f_{r}-1}\right.\\
    &\qquad\cdot\left.e_{u}e_{0}^{l_{s}-1}e_{u/y_{s}}\cdots e_{0}^{l_{1}-1}e_{u/y_{1}};(\beta(1-z),(2z-1)\beta'(1-z))\right)\\
    &=\sum_{0\le l\le j\le n}\sum_{\substack{f_{1},\ldots,f_{r}\ge 0\\ f_{1}+\cdots+f_{r}=n-j}}\sum_{\substack{g_{1},\ldots,g_{r+s}\ge 0\\ g_{1}+\cdots+g_{r+s}=l}}\left(\prod_{a=1}^{r}\binom{k_{a}+f_{a}-1}{f_{a}}\binom{k_{a}+f_{a}+g_{a}-1}{g_{a}}\right)\left(\prod_{a=1}^{s}\binom{l_{a}+g_{r+a}-1}{g_{r+a}}\right)\\
    &\qquad\cdot\frac{(-1)^{l}}{(j-l)!}I_{\beta_{z}}\left((\beta(z),(1-2z)\beta'(z));e_{0}^{j-l}\sh e_{u/x_{1}}e_{0}^{k_{1}+f_{1}+g_{1}-1}\cdots e_{u/x_{r}}e_{0}^{k_{r}+f_{r}+g_{r}-1}\right.\\
    &\qquad\cdot\left.e_{u}e_{0}^{l_{s}+g_{r+s}-1}e_{u/y_{s}}\cdots e_{0}^{l_{1}+g_{r+1}-1}e_{u/y_{1}};(\beta(1-z),(2z-1)\beta'(1-z))\right).
  \end{align}
  Using the shuffle product formula (Proposition \ref{prop_ii_properties} \eqref{shuffle}) and the fact
  \begin{align}
    &I_{\beta_{z}}((\beta(z),(1-2z)\beta'(z));e_{0}^{j-l};(\beta(1-z),(2z-1)\beta'(1-z)))\\
    &=\frac{1}{(j-l)!}(\log\beta(1-z)-\log\beta(z))^{j-l}\xrightarrow{z\to 0}\begin{cases}1 & \text{if }j=l,\\ 0 & \text{if }j\neq l,\end{cases}
  \end{align}
  we obtain
  \begin{align}
    &I_{\beta_{z}}(X''_{u};\bk_{n})\\
    &=\sum_{\substack{f_{1},\ldots,f_{2r+s}\ge 0\\ f_{1}+\cdots+f_{2r+s}=n}}(-1)^{f_{r+1}+\cdots+f_{2r+s}}\left(\prod_{a=1}^{r}\binom{k_{a}+f_{a}-1}{f_{a}}\binom{k_{a}+f_{a}+f_{r+a}-1}{f_{r+a}}\right)\left(\prod_{a=1}^{s}\binom{l_{a}+f_{2r+a}-1}{f_{2r+a}}\right)\\
    &\qquad\cdot I_{\beta_{z}}\left((\beta(z),(1-2z)\beta'(z));e_{u/x_{1}}e_{0}^{k_{1}+f_{1}+f_{r+1}-1}\cdots e_{u/x_{r}}e_{0}^{k_{r}+f_{r}+f_{2r}-1}\right.\\
    &\qquad\cdot\left.e_{u}e_{0}^{l_{s}+f_{2r+s}-1}e_{u/y_{s}}\cdots e_{0}^{l_{1}+f_{2r+1}-1}e_{u/y_{1}};(\beta(1-z),(2z-1)\beta'(1-z))\right)\\
    &=\sum_{\substack{f_{2r+1},\ldots,f_{2r+s}\ge 0\\ f_{2r+1}+\cdots+f_{2r+s}=n}}(-1)^{f_{2r+1}+\cdots+f_{2r+s}}\left(\prod_{a=1}^{s}\binom{l_{a}+f_{2r+a}-1}{f_{2r+a}}\right)I_{\beta_{z}}\left((\beta(z),(1-2z)\beta'(z));e_{u/x_{1}}e_{0}^{k_{1}-1}\cdots e_{u/x_{r}}e_{0}^{k_{r}-1}\right.\\
    &\qquad\cdot\left.e_{u}e_{0}^{l_{s}+f_{2r+s}-1}e_{u/y_{s}}\cdots e_{0}^{l_{1}+f_{2r+1}-1}e_{u/y_{1}};(\beta(1-z),(2z-1)\beta'(1-z))\right).
  \end{align}
  Note that the last integral is converges when $z\to 0$. Combining this result, \eqref{eq_t_shuffle_1} and \eqref{eq_t_shuffle_2}, we have
  \begin{align}
    &\fmpl_{\hcS,\alpha}\left(\vv{x_{1},\ldots,x_{r}}{k_{1},\ldots,k_{r}}\sh\vv{y_{1},\ldots,y_{s}}{l_{1},\ldots,l_{s}}\right)\\
    &\equiv \fmpl_{\hcRS,\alpha}(X;\bk)\qquad(\mod~2\pi i)\\
    &=\frac{(-1)^{r+s}}{2\pi i}\sum_{n=0}^{\infty}\sum_{u\in D}u^{\alpha}I_{\beta}(X'_{u};\bk_{n})t^{n}\\
    &\equiv\frac{(-1)^{r+s+l_{1}+\cdots+l_{s}}}{2\pi i}\sum_{n=0}^{\infty}\sum_{u\in D}u^{\alpha}I_{\beta}(X''_{u};\bk_{n})(-t)^{n}\qquad (\mod~2\pi i)\\
    &=\frac{(-1)^{r+s+l_{1}+\cdots+l_{s}}}{2\pi i}\sum_{n=0}^{\infty}\sum_{u\in D}u^{\alpha}\sum_{\substack{f_{2r+1},\ldots,f_{2r+s}\ge 0\\ f_{2r+1}+\cdots+f_{2r+s}=n}}(-1)^{f_{2r+1}+\cdots+f_{2r+s}}\left(\prod_{a=1}^{s}\binom{l_{a}+f_{2r+a}-1}{f_{2r+a}}\right)\\
    &\qquad\cdot I_{\beta}(\tzero;e_{u/x_{1}}e_{0}^{k_{1}-1}\cdots e_{u/x_{r}}e_{0}^{k_{r}-1}e_{u}e_{0}^{l_{s}+f_{2r+s}-1}e_{u/y_{s}}\cdots e_{0}^{l_{1}+f_{2r+1}-1}e_{u/y_{1}};\tzero)t^{n}\\
    &\equiv (-1)^{l_{1}+\cdots+l_{s}}y_{1}^{\alpha}\sum_{f_{1},\ldots,f_{s}\ge 0}\left(\prod_{j=1}^{s}\binom{l_{j}+f_{j}-1}{f_{j}}\right)\\
    &\qquad\cdot\fmpl_{\hcS,\alpha}\vv{x_{1}/y_{1},\ldots,x_{r}/y_{1},1/y_{1},y_{s}/y_{1},\ldots,y_{2}/y_{1}}{k_{1},\ldots,k_{r},l_{s}+f_{s},\ldots,l_{1}+f_{1}}t^{f_{1}+\cdots+f_{s}}\qquad (\mod~2\pi i).
  \end{align}
  Here we used Propsition \ref{prop_t_smpl} \eqref{t_smpl_3} in the second and last congruence.
\end{proof}
\begin{definition}
  For $\vv{\etab}{\bk}=\vv{\eta_{1},\ldots,\eta_{r}}{k_{1},\ldots,k_{r}}\in\mu_{N}^{r}\times\bbZ_{\ge 1}^{r}$, we define the \emph{$t$-adic symmetric multiple $L$-value} $L_{\hcS,\alpha}\vv{\etab}{\bk}$ as the projection on $(\cZ_{N}/\pi^{2}\cZ_{N})\jump{t}\simeq\cZ_{N}\jump{t}/\pi^{2}\cZ_{N}\jump{t}$ of
      \begin{align}
      L_{\hcS,\alpha}^{\sh}\vv{\etab}{\bk}&\coloneqq\sum_{i=0}^{r}(-1)^{k_{i+1}+\cdots+k_{r}}\eta_{i+1}^{\alpha}L^{\sh}\vv{\eta_{1}/\eta_{i+1},\ldots,\eta_{i}/\eta_{i+1}}{k_{1},\ldots,k_{i}}\\
        &\qquad\cdot\sum_{f_{i+1},\ldots,f_{r}\ge 0}\left(\prod_{j=i+1}^{r}\binom{k_{j}+f_{j}-1}{f_{j}}\right)L^{\sh}\vv{1/\eta_{i+1},\eta_{r}/\eta_{i+1},\ldots,\eta_{i+2}/\eta_{i+1}}{k_{r}+f_{r},\ldots,k_{i+1}+f_{i+1}}t^{f_{i+1}+\cdots+f_{r}}.
      \end{align}
    We also $\bbQ$-linearly extend the map $L_{\hcS,\alpha}$ to $\cR_{N}$.
\end{definition}
\begin{remark}
  Let us define
  \begin{align}
    L_{\hcS,\alpha}^{\ast}\vv{\eta_{1},\ldots,\eta_{r}}{k_{1},\ldots,k_{r}}&\coloneqq\sum_{i=0}^{r}(-1)^{k_{i+1}+\cdots+k_{r}}(\eta_{i+1}\cdots\eta_{r})^{\alpha}L^{\ast}\vv{\eta_{1},\ldots,\eta_{i}}{k_{1},\ldots,k_{i}}\\
      &\qquad\cdot\sum_{f_{i+1},\ldots,f_{r}\ge 0}\left(\prod_{j=i+1}^{r}\binom{k_{j}+f_{j}-1}{f_{j}}\right)L^{\ast}\vv{1/\eta_{r},\ldots,1/\eta_{i+1}}{k_{r}+f_{r},\ldots,k_{i+1}+f_{i+1}}t^{f_{i+1}+\cdots+f_{r}},
    \end{align}
  for positive integers $k_{1},\ldots,k_{r}$ and $\eta_{1},\ldots,\eta_{r}\in\mu_{N}$ by using the harmonic regularization $L^{\ast}\vv{\etab}{\bk}$ (we omit its definition). Then a similar computation to \cite[Proposition 2.1]{osy21} establishes that the value $L_{\hcS,\alpha}\vv{\eta_{1},\ldots,\eta_{r}}{k_{1},\ldots,k_{r}}$ also coincides with the projection on $(\cZ_{N}/\pi^{2}\cZ_{N})\jump{t}$ of $L_{\hcS,\alpha}^{\ast}\vv{\eta_{1}/\eta_{2},\ldots,\eta_{r-1}/\eta_{r},\eta_{r}}{k_{1},\ldots,k_{r}}$ (see also \cite[\S 3]{tasaka21}, \cite[Theorem 4.6]{sz20}).
\end{remark}
The following lemma is proved by the similar argument to \cite[Proposition 2.4]{osy21}.
\begin{lemma}\label{lem_convergence}
  For any $\vv{\etab}{\bk}=\vv{\eta_{1},\ldots,\eta_{r}}{k_{1},\ldots,k_{r}}\in\mu_{N}^{r}\times\bbZ_{\ge 1}^{r}$, the value $\fmpl_{\hcS,\alpha}\vv{\etab}{\bk}$ converges, coefficientwise as the power series of $t$, and coincides with $L_{\hcS,\alpha}\vv{\etab}{\bk}$.
\end{lemma}
\begin{corollary}
  Let $k_{1},\ldots,k_{r},l_{1},\ldots,l_{s}$ be positive integers and $\eta_{1},\ldots,\eta_{r},\xi_{1},\ldots,\xi_{s}$ elements of $\mu_{N}$. Then we have  
  \begin{align}
    &L_{\hcS,\alpha}\left(\vv{\eta_{1},\ldots,\eta_{r}}{k_{1},\ldots,k_{r}}\sh\vv{\xi_{1},\ldots,\xi_{s}}{l_{1},\ldots,l_{s}}\right)\\
    &\qquad=(-1)^{l_{1}+\cdots+l_{s}}\xi_{1}^{\alpha}\sum_{f_{1},\ldots,f_{s}\ge 0}\left(\prod_{i=1}^{s}\binom{l_{i}+f_{i}-1}{f_{i}}\right)L_{\hcS,\alpha}\vv{\eta_{1}/\xi_{1},\ldots,\eta_{r}/\xi_{1},1/\xi_{1},\xi_{s}/\xi_{1},\ldots,\xi_{2}/\xi_{1}}{k_{1},\ldots,k_{r},l_{s}+f_{s},\ldots,l_{1}+f_{1}}t^{f_{1}+\cdots+f_{s}}.
\end{align} 
\end{corollary}
\begin{proof}
  Immediately follows from Theorem \ref{thm_t_shuffle} and Lemma \ref{lem_convergence}. 
\end{proof}

\end{document}